\documentclass[12pt]{amsart}
\usepackage{amscd,amsmath,amsthm,amssymb}
\usepackage[left]{lineno}
\usepackage{color}
\usepackage{stmaryrd}
\usepackage[utf8]{inputenc}
\usepackage{cleveref}
\usepackage{epstopdf}
\usepackage{graphicx}
\definecolor{verylight}{gray}{0.97}
\definecolor{light}{gray}{0.9}
\definecolor{medium}{gray}{0.85}
\definecolor{dark}{gray}{0.6}

 \usepackage{fullpage}

 \def\NZQ{\mathbb}               

 \def\KK{{\NZQ K}}
 
 \def\G{{\mathcal G}}

 \def\0b{{\mathbf 0}}

\newcommand\calP{\mathcal{P}}
 \def\opn#1#2{\def#1{\operatorname{#2}}} 
 \opn\chara{char}
 \opn\length{\ell}
 \opn\pd{pd}
 \opn\rk{rk}
 \opn\projdim{proj\,dim}
 \opn\injdim{inj\,dim}
 \opn\rank{rank}
 \opn\depth{depth}
 \opn\grade{grade}
 \opn\height{height}
 \opn\embdim{emb\,dim}
 \opn\codim{codim}
 
 \opn\Tr{Tr}
 \opn\bigrank{big\,rank}
 \opn\superheight{superheight}
 \opn\lcm{lcm}
 \opn\trdeg{tr\,deg}
 \opn\reg{reg}
 \opn\lreg{lreg}
 \opn\ini{in}
 \opn\lpd{lpd}
 \opn\size{size}
 \opn\sdepth{sdepth}
 \opn\link{link}
 \opn\fdepth{fdepth}\opn\lex{lex}
 \opn\tr{tr}
 \opn\type{type}
 \opn\gap{gap}
 \opn\arithdeg{arith-deg}
 \opn\HS{HS}
 \opn\GL{GL}
 \opn\supp{supp}

 \opn\div{div} \opn\Div{Div} \opn\cl{cl} \opn\Cl{Cl}
 \opn\Spec{Spec} \opn\Supp{Supp} \opn\supp{supp} \opn\Sing{Sing}
 \opn\Ass{Ass} \opn\Min{Min}\opn\Mon{Mon}
 \opn\Ann{Ann} \opn\Rad{Rad} \opn\Soc{Soc}\opn\Deg{Deg}
 \opn\Im{Im} \opn\Ker{Ker} \opn\Coker{Coker} \opn\Am{Am}
 \opn\Hom{Hom} \opn\Tor{Tor} \opn\Ext{Ext} \opn\End{End}
 \opn\Aut{Aut} \opn\id{id}
 
 \opn\nat{nat}
 \opn\pff{pf}
 \opn\Pf{Pf} \opn\GL{GL} \opn\SL{SL} \opn\mod{mod} \opn\ord{ord}
 \opn\Gin{Gin} \opn\Hilb{Hilb}\opn\sort{sort}
 \opn\PF{PF}\opn\Ap{Ap}
 \opn\mult{mult}
 \opn\bight{bight}
 \opn\aff{aff}
 \opn\relint{relint} \opn\st{st}
 \opn\lk{lk} \opn\cn{cn} \opn\core{core} \opn\vol{vol}  \opn\inp{inp} \opn\nilpot{nilpot}
 \opn\link{link} \opn\star{star}\opn\lex{lex}\opn\set{set}
 \opn\width{wd}
 \opn\Fr{F}
 \opn\QF{QF}
 \opn\G{G}
 \opn\type{type}\opn\res{res}
 \opn\conv{conv}
 \opn\Ind{Ind}
 \opn\gr{gr}
 
 \def\pot#1#2{#1[\kern-0.28ex[#2]\kern-0.28ex]}

 \opn\dirlim{\underrightarrow{\lim}}
 \opn\inivlim{\underleftarrow{\lim}}
 \def\Implies{\ifmmode\Longrightarrow \else
         \unskip${}\Longrightarrow{}$\ignorespaces\fi}
 \def\implies{\ifmmode\Rightarrow \else
         \unskip${}\Rightarrow{}$\ignorespaces\fi}
 \def\iff{\ifmmode\Longleftrightarrow \else
         \unskip${}\Longleftrightarrow{}$\ignorespaces\fi}

 \let\:=\colon
 \newtheorem{Theorem}{Theorem}[section]
 \newtheorem{Lemma}[Theorem]{Lemma}
 \newtheorem{Corollary}[Theorem]{Corollary}
 \newtheorem{Proposition}[Theorem]{Proposition}
 \newtheorem{Remark}[Theorem]{Remark}
 
 \newtheorem{Example}[Theorem]{Example}
 
 \newtheorem{Definition}[Theorem]{Definition}

 \let\epsilon\varepsilon
 \let\kappa=\varkappa
 \def\qed{\ifhmode\textqed\fi
       \ifmmode\ifinner\quad\qedsymbol\else\dispqed\fi\fi}
 \def\textqed{\unskip\nobreak\penalty50
        \hskip2em\hbox{}\nobreak\hfil\qedsymbol
        \parfillskip=0pt \finalhyphendemerits=0}
 \def\dispqed{\rlap{\qquad\qedsymbol}}
 
 \opn\dis{dis}
 \def\pnt{{\raise0.5mm\hbox{\large\bf.}}}
 
 \opn\Lex{Lex}

\begin{document}

\title{Edge ideals of some edge-weighted graphs}

\author{Guangjun Zhu$^{\ast}$,  Shiya Duan, Yijun Cui and Jiaxin Li}


\address{School of Mathematical Sciences, Soochow University, Suzhou, Jiangsu, 215006, P. R. China}

\email{zhuguangjun@suda.edu.cn(Corresponding author:Guangjun Zhu),\linebreak[4]
3136566920@qq.com(Shiya Duan), 237546805@qq.com(Yijun Cui),\linebreak[4]
lijiaxinworking@163.com(Jiaxin Li).}



\maketitle
\begin{abstract}
This paper presents exact formulas for the regularity and depth of powers of edge ideals of an edge-weighted star graph. Additionally, we provide exact formulas for the regularity of powers of the edge ideal of an edge-weighted integrally closed path, as well as lower bounds on the depth of powers of such an edge ideal.
\end{abstract}

\section{Introduction}
In this article, a graph means a simple graph without loops, multiple edges, and isolated vertices. Let $G$ be a graph with vertex set $V(G)=\{x_1,\ldots,x_n\}$ and   edge set $E(G)$. Suppose $w: E(G)\rightarrow \mathbb{Z}_{>0}$ is an edge weight function on $G$. We write $G_\omega$ for the pair $(G,\omega)$ and call it an {\em edge-weighted} graph with the underlying graph  $G$.
For a weighted graph $G_\omega$, its {\em edge-weighted ideal}  (or simply edge ideal), was introduced in \cite{PS}, is the ideal of the polynomial ring $S=\KK[x_{1},\dots, x_{n}]$ in $n$ variables over a field $\KK$ given by
\[
I(G_\omega)=(x_i^{\omega(e)}x_j^{\omega(e)}\mid e:=\{x_i,x_j\}\in E(G_\omega)).
\]
If $w$ is the constant function defined by $w(e)=1$ for $e\in E(G)$, then  $I(G_\omega)$ is the classical  edge ideal of the underlying graph $G$ of $G_\omega$, which
has been extensively studied in the literature \cite{B,BHT,MRW,FM,HT,M,MV,W,Z}.

Recently, there has been a surge of interest in characterizing weights for which the edge ideals of edge-weighted graphs are Cohen-Macaulay. For example,
Paulsen and Sather-Wagstaff in \cite{PS}  classified Cohen-Macalay edge-weighted graphs $G_\omega$ where the underlying
graph $G$ is a  cycle, a tree, or a  complete graph.  Seyed Fakhari et al. in \cite{FSTY} continued this study, they  classified
Cohen-Macalay edge-weighted graph $G_\omega$ when $G$ is a very well-covered graph. Recently, Diem et al. in \cite{DMV} gave a  complete  characterization of sequentially Cohen-Macaulay edge-weighted graphs. In \cite{W},  Wei classified all Cohen-Macaulay weighted chordal graphs from a purely graph-theoretic point of view. Hien in \cite{Hi}
classified Cohen-Macaulay edge-weighted graphs $G_\omega$ when $G$ has girth at least $5$. 

 Integral closure and normality of monomial ideals is also an interesting topic.
In \cite{DZCL}, we   gave a  complete  characterization of an integrally closed edge-weighted graph $G_\omega$ and  showed that if its underlying
graph $G$ is a  star graph,  a  path, or a cycle, then $G_\omega$ is normal.

The study of edge ideals of edge-weighted graphs is much more recent and consequently there are  fewer results in this direction.
In this paper, we decide to focus on the regularity  and depth of powers of the edge ideal $I(G_\omega)$, where   $G_\omega$ is a special graph. Recall that the regularity and depth are two central invariants associated to a homogeneous ideal  $I$. It is well known that
 $\reg(I^t)$ is asymptotically a linear function for $t\gg  0$, i.e., there exist constants $a$,  $b$ and a positive integer $t_0$ such that for all $t\geq t_0$, $\mbox{reg}\,(I^t)=at+b$ (see \cite{CHT,K}). In this regard, there has been of interest to find  the exact form of this linear function and to determine  the stabilization index $t_0$ at  which $\reg(I^t)$ becomes linear (cf. \cite{B,BHT}). It turns out that even in the case of monomial ideals it is challenging to find the linear function and $t_0$ (see \cite{Con}).  In \cite{Br}, Brodmann  showed that  $\mbox{depth}\,(S/I^t)$ is a constant for $t\gg 0$, and that this constant is bounded
above by $n-\ell(I)$, where $\ell(I)$ is the analytic spread of $I$.  In this regard, there has been an interest in determining the smallest value $t_0$ such that $\mbox{depth}\,(S/I^t)$ is a constant for all $t\geq t_ 0$ (see \cite{FM,HHi,M}).

The article is organized as follows. In Section \ref{sec:prelim}, we  provides a review of  important definitions and
terminology will be necessary later.
In Section \ref{sec:star},  by choosing different exact sequences and repeatedly using Lemma \ref{exact},  we  give some exact formulas for the regularity and  depth  of
powers of the edge ideal of an edge-weighted star graph. In Section \ref{sec:path}, using Betti splitting and polarization approaches, we  give some exact formulas for the regularity of
powers of the edge ideal of an edge-weighted integrally closed   path. We also provide some lower bounds on the depth of powers of such an edge ideal.

\section{Preliminaries}
\label{sec:prelim}

In this section, we provide the definitions and basic facts which will be used throughout this paper.
We refer to \cite{BH} and \cite{HH} for detailed information.

\subsection{Notions of simple graphs}
 Let $G$ be a simple graph with the vertex set $V(G)$ and the edge set $E(G)$. For any subset $A$ of $V(G)$, the \emph{induced subgraph} of $G$ on the set $A$, denoted by $G[A]$, satisfies that $V(G[A])=A$ and for any $x_i,x_j \in A$, $\{x_i,x_j\} \in E(G[A])$ if and only if $\{x_i,x_j\}\in E(G)$. At the same time, the induced subgraph of $G$ on the set $V(G)\setminus A$ will be denoted by $G\setminus A$. In particular, if $A=\{v\}$  then we will write $G\setminus v$ instead of $G\setminus \{v\}$ for simplicity. For any vertex $v\in V(G)$,
its \emph{neighborhood} is defined as $N_G(v)\!:=\{u \in V(G)\mid \{u,v\}\in E(G)\}$.

An edge-weighted  graph is  called a \emph{non-trivially  weighted} graph if there is at least one edge with  a weight  greater than $1$. Otherwise, it is called a
{\em trivially weighted} graph. An edge $e \in E(G_\omega)$ with {\em non-trivially weight} if $w(e) \ge2$. Otherwise, we say $e$ with {\em trivial weight}. 

A \emph{walk} $W$ of length $n-1$ in a graph $G$ is a sequence of vertices $w_1$ through $w_{n}$, where each consecutive pair of vertices $\{w_i,w_{i+1}\}$ is connected by an edge in $G$.
A \emph{path} is a walk where all vertices are distinct, and a \emph{cycle} is a walk where $w_1=w_n$ and other vertices are distinct. To simplify notation, a path of length $n-1$ is denoted by $P_n$, and a cycle of length $n$ is denoted by $C_n$.

\subsection{Notions from commutative algebra}

For any homogeneous ideal $I$ of the polynomial ring $R=\KK[z_{1},\ldots,z_{n}]$, there exists a \emph{graded minimal free resolution}
\[
 0\rightarrow \bigoplus\limits_{j}R(-j)^{\beta_{p,j}(R/I)}\rightarrow \bigoplus\limits_{j}R(-j)^{\beta_{p-1,j}(R/I)}\rightarrow \cdots\rightarrow \bigoplus\limits_{j}R(-j)^{\beta_{0,j}(R/I)}\rightarrow R/I\rightarrow 0,
\]
where $p\leq n$ and $R(-j)$ is obtained from $R$ by a shift of degree $j$. The number $\beta_{i,j}(R/I)$, the $(i,j)$-th graded \emph{Betti} number of $R/I$, is an invariant of $R/I$ that equals the number of minimal generators of degree $j$ in the $i$-th syzygy module of $R/I$. Of particular interest is the following invariant which measures the “size” of the minimal graded free resolution of $R/I$. The regularity of $R/I$, denoted $\reg(R/I)$, is defined by
\[
    \reg(R/I)\!:=\max\{j-i\mid \beta_{i,j}(R/I)\neq 0\}.
\]
Meanwhile, the \emph{projective dimension} of $R/I$, denoted by $\pd(R/I)$, is
\[
    \pd(R/I)\!:=\max\{i\mid \beta_{i,j}(R/I)\neq 0\}.
\]
These two invariants measure the complexity of the minimal graded free resolution of $R/I$.

\vspace{3mm}For  a monomial ideal $I$,   let $\mathcal{G}(I)$ denote its unique minimal set
of monomial generators. We now derive some formulas for $\mbox{pd}\,(I)$  and $\mbox{reg}\,(I)$ in some special cases by using some
tools developed in \cite{FHT}.

\begin{Definition} \label{bettispliting}Let $I$  be a monomial ideal. If there exist  monomial
ideals $J$ and $K$ such that $\mathcal{G}(I)$ is the disjoint union of $\mathcal{G}(J)$ and $\mathcal{G}(K)$. Then $I=J+K$
is a {\em Betti splitting} if
$$\beta_{i,j}(I)=\beta_{i,j}(J)+\beta_{i,j}(K)+\beta_{i-1,j}(J\cap K)\hspace{2mm}\mbox{for all}\hspace{2mm}i,j\geq 0,$$
where $\beta_{i-1,j}(J\cap K)=0\hspace{2mm}  \mbox{if}\hspace{2mm} i=0$.
\end{Definition}

 Definition \ref{bettispliting} implies the following results.
\begin{Corollary} \label{cor1}
If $I=J+K$ is a Betti splitting ideal, then
\begin{itemize}
 \item[(1)]$\reg(I)=\mbox{max}\,\{\reg(J),\reg(K),\reg(J\cap K)-1\}$,
 \item[(2)] $\pd(I)=\mbox{max}\,\{\pd(J),\pd(K),\pd(J\cap K)+1\}$.
\end{itemize}
\end{Corollary}

\medskip
This formula was first obtained for the total Betti numbers by
Eliahou and Kervaire \cite{EK} and extended to the graded case by Fatabbi \cite{F}.
In  \cite{FHT}, the authors describe some sufficient conditions for an
ideal $I$ to have a Betti splitting.

\begin{Lemma}
\label{spliting}{\em (\cite[Corollary 2.7]{FHT})}
Suppose that $I=J+K$ where $\mathcal{G}(J)$ contains all
the generators of $I$ divisible by some variable $x_{i}$ and $\mathcal{G}(K)$ is a nonempty set containing
the remaining generators of $I$. If $J$ has a linear resolution, then $I=J+K$ is a Betti
splitting.
\end{Lemma}

The following three lemmas are often used in this article.
\begin{Lemma}
\label{quotient} {\em\cite[Lemma 1.3]{HT1})} Let $S=\KK[x_{1},\dots, x_{n}]$ be a polynomial ring over a field $\KK$ and let $I$ be a proper non-zero homogeneous
ideal in $S$. Then
\begin{itemize}
\item[(1)] $\pd(I)=\pd(S/I)-1$,
\item[(2)] $\reg(I)=\reg(S/I)+1$.
\end{itemize}
\end{Lemma}
In particular, if $u$ is a monomial of degree $d$ in $S$ and  $K=(u)$, then we have $\reg(S/K)=d-1$.

 By Auslander-Buchsbaum formula (see \cite[Corollary A.4.3]{HH}), we have 
\[
\depth(S/I)=n-\pd(S/I).
\]

\begin{Lemma}{\em (\cite[Lemma 2.2 and  Lemma 3.2 ]{HT2})}
\label{sum2}
Let $S_{1}=\KK[x_{1},\dots,x_{m}]$, $S_{2}=\KK[x_{m+1},\dots,x_{n}]$ and $S=\KK[x_{1},\dots,x_{n}]$ be three polynomial rings over $\KK$, $I\subseteq S_{1}$ and $J\subseteq S_{2}$ be two proper non-zero homogeneous  ideals. Then we have
\begin{itemize}
\item[(1)] $\reg(S/(I+J))=\reg(S_{1}/I)+\reg(S_{2}/J)$,
\item[(2)] $\depth(S/(I+J))=\depth(S_{1}/I)+\depth(S_{2}/J)$.
\end{itemize}
\end{Lemma}

\begin{Lemma}  {\em (\cite[Lemmas 2.1 and 3.1]{HT2})}
	\label{exact}
	Let $0\longrightarrow M\longrightarrow N\longrightarrow P\longrightarrow 0$ be a short exact
	sequence of finitely generated graded S-modules. Then we have
	\begin{itemize}
		\item[(1)]$\reg\,(N)\leq max\{\reg\,(M), \reg\,(P)\}$, the equality holds if $\reg\,(P) \neq \reg\,(M)-1$.
		\item[(2)]$\depth\,(N)\geq min\{\depth\,(M), \depth\,(P)\}$, the equality holds if $\depth\,(P) \neq \depth\,(M)-1$.
	\end{itemize}
\end{Lemma}

\section{Star graph}
\label{sec:star}

In this section, we will give precise formulas for the depth and  regularity of powers of the edge ideal of an edge-weighted star graph. For a positive integer $n$, the notation $[n]$ denotes the set $\{1,2,\dots,n\}$. 

\begin{Theorem} \label{star}
Let $G_\omega$ be an edge-weighted star graph with $n$ vertices, and let  the set of monomial generators of its edge ideal be $\mathcal{G}(I(G_\omega))=\{(x_ix_n)^{\omega_i}\mid i\in [n-1]\}$.
 Then 
\begin{itemize}
\item[(1)]$\depth(S/I(G_\omega))=1$.
\item[(2)]$\reg(S/I(G_\omega))=\omega+\sum\limits_{i=1}\limits^{n-1}{(\omega_i-1)}$, where $\omega=\max\,\{\omega_1,\ldots,\omega_{n-1}\}$.
\end{itemize}
\end{Theorem}
\begin{proof} Let us assume, without loss of generality, that  $\omega_1 \geq \omega_2 \geq \cdots \geq \omega_{n-1}$. We will now proceed to prove the given  statements   by induction on $n$, we first establish the base case where $n=2$, which is trivial. 
For $n=3$,  the following equalities hold: $I(G_\omega) : x_{3}^{\omega_{2}}=(x_1^{\omega_1}x_3^{\omega_1-\omega_2},x_2^{\omega_2})$, $(I(G_\omega), x_{3}^{\omega_{2}})=(x_3^{\omega_2})$, $I(G_\omega) : x_{2}^{\omega_{2}}=(x_3^{\omega_2})$, and $(I(G_\omega), x_{2}^{\omega_{2}})=((x_1x_3)^{\omega_1},x_2^{\omega_2})$. Therefore, $\depth(S/(I(G_\omega) \colon x_{3}^{\omega_{2}}))=1$, $\depth(S/(I(G_\omega), x_{3}^{\omega_{2}}))=2$,   $\reg(S/(I(G_\omega) \colon x_{2}^{\omega_{2}}))=\omega_{2}-1$, and  $\reg(S/(I(G_\omega), x_{2}^{\omega_{2}}))=2\omega_{1}+\omega_{2}-2$. By analyzing the short exact sequences
\begin{gather*}
\begin{matrix}
 0 & \rightarrow & \frac{S}{I(G_\omega) : x_3^{\omega_2}}(-\omega_2)  &  \stackrel{\cdot x_{3}^{\omega_{2}}} \longrightarrow  & \frac{S}{I(G_\omega)} & \rightarrow &  \frac{S}{(I(G_\omega),x_3^{\omega_2})} & \rightarrow & 0,\\
 0 & \rightarrow & \frac{S}{I(G_\omega) : x_2^{\omega_2}}(-\omega_2)  &  \stackrel{\cdot x_{2}^{\omega_{2}}} \longrightarrow  & \frac{S}{I(G_\omega)} & \rightarrow &  \frac{S}{(I(G_\omega),x_2^{\omega_2})} & \rightarrow & 0,
 \end{matrix}
\end{gather*}
using Lemma \ref{exact},  we can obtain that $\depth(S/I(G_\omega))=1$ and $\reg(S/I(G_\omega))=2\omega_{1}+\omega_{2}-2$.

In the following, we assume that  $n \geq 4$ and that the results hold  for $n-1$.  Since  $I(G_\omega) : x_{n-1}^{\omega_{n-1}}=(x_n^{\omega_{n-1}})$ and $(I(G_\omega), x_{n-1}^{\omega_{n-1}})=J+(x_{n-1}^{\omega_{n-1}})$ where  $J=((x_1x_n)^{\omega_1},\ldots,(x_{n-2}x_n)^{\omega_{n-2}})$, we obtain $\depth\Big(\frac{S}{I(G_\omega) \colon x_{n-1}^{\omega_{n-1}}}\Big)=n-1$, $\depth\Big(\frac{S}{(I(G_\omega), x_{n-1}^{\omega_{n-1}})}\Big)
\\ =\depth(\frac{S'}{J})=1$. Furthermore, $\reg\Big(\frac{S}{I(G_\omega) \colon x_{n-1}^{\omega_{n-1}}}\Big)
=\omega_{n-1}-1$ and $\reg\left(\frac{S}{(I(G_\omega), x_{n-1}^{\omega_{n-1}})}\right)=\reg(\frac{S'}{J})+(\omega_{n-1}-1)=\omega_1+\sum\limits_{i=1}\limits^{n-1}{(\omega_i-1)}$ by the inductive hypothesis, where $S'=\KK[x_1,\ldots,x_{n-2},x_n]$.  Applying Lemma  \ref{exact}   to the following short exact sequence
\begin{gather*}
\begin{matrix}
 0 & \rightarrow & \frac{S}{I(G_\omega) : x_{n-1}^{\omega_{n-1}}}(-\omega_{n-1})  &  \stackrel{\cdot x_{n-1}^{\omega_{n-1}}} \longrightarrow  & \frac{S}{I(G_\omega)} & \rightarrow &  \frac{S}{(I(G_\omega),x_{n-1}^{\omega_{n-1}})} & \rightarrow & 0,
 \end{matrix}
\end{gather*}
we get that $\depth(S/I(G_\omega))=1$ and  $\reg(S/I(G_\omega))=\omega_1+\sum\limits_{i=1}\limits^{n-1}{(\omega_i-1)}$.$\hfill$
\end{proof}

\begin{Lemma}\label{starcolon}
Let $G_\omega$ be an edge-weighted star graph as in Theorem \ref{star}. Suppose $\omega_1 \geq \omega_2 \geq \cdots \geq \omega_{n-1}$. Then,  for any $t\ge 2$, we have
\begin{itemize}
		\item[(1)] $(I(G_\omega)^t : (x_{n-1}x_{n})^{\omega_{n-1}})=I(G_\omega)^{t-1}$;
 \item[(2)] $((I(G_\omega)^t : x_{n-1}^{\omega_{n-1}}),x_n^{\omega_{n-1}})=(x_n^{\omega_{n-1}})$;
 \item[(3)] $(I(G_\omega)^t,x_{n-1}^{\omega_{n-1}})=I((G_\omega\setminus {x_{n-1}})^t,x_{n-1}^{\omega_{n-1}})$.
 \end{itemize}
  \end{Lemma}
\begin{proof}
(1) For any monomial  $u$ in $\mathcal{G}(I(G_\omega)^t : (x_{n-1}x_n)^{\omega_{n-1}})$, we have  $u(x_{n-1}x_n)^{\omega_{n-1}} \in I(G_\omega)^t$. Let $u(x_{n-1}x_n)^{\omega_{n-1}}=u_{i1} \cdots u_{it}h$, where each $u_{ij} \in \mathcal{G}(I(G_\omega))$ and $h$ is a monomial.
 If there exists some  $j\in [t]$ such that $x_{n-1}|u_{ij}$, then $u_{ij}=(x_{n-1}x_n)^{\omega_{n-1}}$
 because $N_G(x_{n-1})=\{x_n\}$.
This implies that $u\in I(G_\omega)^{t-1}$. If $x_{n-1}\nmid u_{ij}$  for all $j\in [t]$, then  $x_{n-1}|h$. Thus , it can be concluded that
$u\in I(G_\omega)^{t-1}$, since $x_n^{\omega_{n-1}}|u_{ij}$  for all $j\in [t]$.

(2) For any monomial $u \in \mathcal{G}(I(G_\omega)^t : x_{n-1}^{\omega_{n-1}})$, we have $ux_{n-1}^{\omega_{n-1}}=u_{i1}\cdots u_{it}h$ for some monomial $h$, where each $u_{ij} \in \mathcal{G}(I(G_\omega))$. Let $u_{ij}$
be represented as $u_{ij}=(x_{ij}x_n)^{\omega_{ij}}$ with $ij \in [n-1]$, then  it can be inferred that $x_n^{\omega_{ij}}|u$. Therefore, $x_n^{\omega_{n-1}}|u$, since $\omega_{ij} \ge \omega_{n-1}$. This forces  that $u\in (x_n^{\omega_{n-1}})$.

(3) It is clear that $(I(G_\omega\setminus {x_{n-1}})^t,x_{n-1}^{\omega_{n-1}}) \subseteq (I(G_\omega)^t,x_{n-1}^{\omega_{n-1}})$. If the  monomial $u \in \mathcal{G}(I(G_\omega)^t) \setminus \mathcal{G}(I(G_\omega\setminus {x_{n-1}})^t)$, then $x_{n-1}|u$. It follows that $(x_{n-1}x_n)^{\omega_{n-1}}|u$, since $N_{G}(x_{n-1})=\{x_n\}$. This implies that $u \in (x_{n-1}^{\omega_{n-1}})$.$\hfill$
\end{proof}

\begin{Theorem}\label{starpower}
Let $G_\omega$ be an edge-weighted star graph as in Lemma \ref{starcolon}. Then,   for $t\ge 2$, we have
\begin{itemize}
\item[(1)]$\depth(S/I(G_\omega)^t)=1$.
\item[(2)]$\reg(S/I(G_\omega)^t)=2(t-1)\omega_1+\reg(S/I(G_\omega))$.
\end{itemize}
\end{Theorem}

\begin{proof} Let $I=I(G_\omega)$. We will prove the assertions  by induction on $n$ and $t$.
 The case where  $n=2$  is trivial.   
 
 In the following, we assume that $n \geq 3$ and that the results hold for $n-1$ and $t-1$.
Using Lemma \ref{starcolon}, Lemma \ref{sum2}, Theorem  \ref{star} and the inductive hypothesis, we can conclude that
\begin{align*}
\depth(S/((I^t : x_{n-1}^{\omega_{n-1}}) : x_n^{\omega_{n-1}}))&=\depth(S/I^{t-1})=1,\\
\depth(S/((I^t : x_{n-1}^{\omega_{n-1}}),x_n^{\omega_{n-1}}))&=\depth(S/(x_n^{\omega_{n-1}}))=n-1,\\
\depth(S/(I^t,x_{n-1}^{\omega_{n-1}}))&=\depth(S/I((G_\omega\setminus {x_{n-1}})^t,x_{n-1}^{\omega_{n-1}}))=1,\\
\reg(S/((I^t : x_{n-1}^{\omega_{n-1}}) : x_n^{\omega_{n-1}}))&=\reg(S/I^{t-1})=2(t-2)\omega_1+\reg(S/I(G_\omega)),\\
\reg(S/((I^t : x_{n-1}^{\omega_{n-1}}),x_n^{\omega_{n-1}}))&=\reg(S/(x_n^{\omega_{n-1}}))=\omega_{n-1}-1,\\
 \reg(S/(I^t,x_{n-1}^{\omega_{n-1}}))&=\reg(S/I((G_\omega\setminus {x_{n-1}})^t,x_{n-1}^{\omega_{n-1}}))\\
 &=2(t-1)\omega_1+\reg(S/I(G_\omega)).
\end{align*}
The desired results hold by   Lemma \ref{exact} and  the following  short exact sequences
\begin{gather*}
\begin{matrix}
 0 & \rightarrow &\frac{S}{I^t : x_{n-1}^{\omega_{n-1}}}(-\omega_{n-1})  & \stackrel{ \cdot x_{n-1}^{\omega_{n-1}}} \longrightarrow  &\frac{S}{I^t} & \rightarrow & \frac{S}{(I^t,x_{n-1}^{\omega_{n-1}})} & \rightarrow & 0,\\
  0 & \rightarrow & \frac{S}{(I^t : x_{n-1}^{\omega_{n-1}}) : x_n^{\omega_{n-1}}}(-\omega_{n-1}) &  \stackrel{ \cdot  x_{n}^{\omega_{n-1}}}  \rightarrow & \frac{S}{I^t : x_{n-1}^{\omega_{n-1}}} &\rightarrow &  \frac{S}{((I^t : x_{n-1}^{\omega_{n-1}}),x_n^{\omega_{n-1}})} & \rightarrow & 0.
 \end{matrix}
  \end{gather*}
\end{proof}

\section{path graph}
\label{sec:path}

This section provides precise formulas for the regularity of powers of the edge ideal of an edge-weighted integrally closed path using Betti splitting and polarization approaches. Additionally, it offers lower bounds on the depth of powers of this edge ideal. The section begins by defining polarization.

\begin{Definition} \label{polarization} {\em (\cite[Definition 2.1]{SF})}
Let $I\subset S$ be a monomial ideal with $\mathcal{G}(I)=\{u_1,\ldots,u_m\}$ where $u_i=\prod\limits_{j=1}^n x_j^{a_{ij}}$ for $i=1,\ldots,m$.
The polarization of $I$, denoted by $I^{\mathcal{P}}$, is a squarefree monomial ideal in the polynomial ring $S^{\mathcal{P}}$
$$I^{\mathcal{P}}=(\mathcal{P}(u_1),\ldots,\mathcal{P}(u_m))$$
where $\mathcal{P}(u_i)=\prod\limits_{j=1}^n \prod\limits_{k=1}^{a_{ij}} x_{jk}$ is a squarefree monomial  in $S^{\mathcal{P}}=\KK[x_{j1},\ldots,x_{ja_j}\mid j=1,\ldots,n]$ and $a_j=\max\{a_{ij}| i=1,\ldots,m\}$ for  $1\leq j\leq n$.
\end{Definition}

A monomial ideal  and its polarization  share numerous homological and
algebraic properties.  The following is a  useful property of polarization.

\begin{Lemma}
\label{polar}{\em (\cite[Corollary 1.6.3]{HH})} Let $I\subset S$ be a monomial ideal and $I^{\calP}\subset S^{\calP}$ be its polarization.
Then
\begin{itemize}
\item[(1)] $\beta_{ij}(I)=\beta_{ij}(I^{\calP})$ for all $i$ and $j$,
\item[(2)] $\reg(I)=\reg(I^{\calP})$,
\item[(3)] $\pd(I)=\pd(I^{\calP})$.
\end{itemize}
\end{Lemma}

\begin{Definition}{\em (\cite[Definition 1.4.1]{HH})}
\label{integrally closed_1} Let  $I$  be an ideal in a ring $R$.
    An element $f \in R$ is  said to be {\em integral} over $I$ if there exists an equation
    \[
    f^k+c_1f^{k-1}+\dots+c_{k-1}f+c_k=0 \text{\ \ with\ \ }c_i \in I^i.
    \]
   The set $\overline{I}$ of elements in $R$ which are integral over $I$  is the \emph{integral closure} of $I$.
If $I=\overline{I}$, then $I$  is said to be  {\em integrally closed}.
    An edge-weighted graph $G_\omega$ is said to be  {\em integrally closed} if its edge ideal $I(G_\omega)$ is integrally closed.
    \end{Definition}

According to \cite[Theorem 1.4.6]{HH}, every edge-weighted graph $G_\omega$ with trivial weights is integrally closed. The following lemma offers a complete characterization of a non-trivially edge-weighted graph that is integrally closed.

 \begin{Lemma}{\em (\cite[Theorem 3.6]{DZCL})}
    \label{integral}
    	 If  $G_\omega$ is a  non-trivially edge-weighted graph, then $I(G_\omega)$ is  integrally closed if and only if  $G_\omega$  does not contain  any of the following three graphs as induced subgraphs.
    \begin{enumerate}
    \item  A path $P_\omega^2$ of length $2$ where  all edges have non-trivial weights.
    \item The  disjoint union $P_\omega^2\sqcup P_\omega^2$ of two paths $P_\omega^2$ where all edges have non-trivial weights.
    \item A $3$-cycle $C_\omega^3$ where all edges have non-trivial weights.
\end{enumerate}
\end{Lemma}

From the lemma above, we can derive
\begin{Corollary} \label{integ}
Let $P_{\omega}^n$ be a  non-trivially weighted integrally closed path  with $n$ vertices, then then it can have  at most two edges with non-trivial weights.
\end{Corollary}

The following lemma provides exact formulas for the regularity and depth of powers of the edge ideal of integrally closed paths with trivial weights.
\begin{Lemma}{\em (\cite[Lemma 2.8]{SM}, \cite[Corollary 7.7.34]{SJ}, \cite[Theorem 4.7 and Remark 2.12]{BHT}, and \cite[Theorem 3.4]{MTV}}{\em )}
\label{trivialpath}
Let $P_{\omega}^n$ be a trivially weighted path  with $n$ vertices, then the following results hold:
 \begin{enumerate}
   \item  $\depth(S/I(P_{\omega}^n))=\lceil \frac{n}{3} \rceil$.
   \item   $\depth(S/I(P^n_\omega)^t)=\max\{\lceil \frac{n-t+1}{3} \rceil,1\}$.
   \item   $\reg(S/I(P_{\omega}^n))=\lfloor \frac{n+1}{3} \rfloor$.
 \item   $\reg(S/I(P^n_\omega)^t)=\lfloor \frac{n+1}{3} \rfloor+2(t-1)=\reg(S/I(P^n_\omega))+2(t-1)$.
 \end{enumerate}
\end{Lemma}

Thus, we will now consider the non-trivial edge-weighted integrally closed path that satisfies the following conditions
\begin{Remark}
    \label{n-path} 
    Let  $n\ge 2$ be an integer and $P_{\omega}^n$  be a non-trivial  edge-weighted integrally closed  path  with the vertex set   $\{x_1,\ldots,x_n\}$ and the edge set  $\{e_1,\ldots,e_{n-1}\}$, where $e_i=\{x_i,x_{i+1}\}$ and $\omega_i=\omega(e_i)$ for all $i \in [n-1]$.
\end{Remark}

We begin by calculating the regularity  and  depth of  the edge ideal of a path $P_{\omega}^n$  with $n\le 4$.

\begin{Theorem}
\label{smalln} Let $P_{\omega}^n$ be a path as in Remark \ref{n-path}, where $n\le 4$. Let $\omega=\max\{\omega_1,\ldots,\omega_{n-1}\}$. Then 
\begin{itemize}
\item[(1)] If $n=2$, then $\reg (S/I(P_{\omega}^n))=2\omega-1$  and  $\depth (S/I(P_{\omega}^n))=1$.
\item[(2)] If $n=3$, then  $\reg (S/I(P_{\omega}^n))=2\omega-1$ and $\depth (S/I(P_{\omega}^n))=1$.
\item[(3)] If $n=4$, then  $\reg (S/I(P_{\omega}^n))=2\omega-1$ and $\depth (S/I(P_{\omega}^n))=2-a$, where  $a=\begin{cases}
            0,             & \text{if $\omega_2=1$,}        \\
            1, & \text{otherwise.}
        \end{cases}$
\end{itemize}		
\end{Theorem}
\begin{proof}  Let $I=I(P_{\omega}^n)$. The  case where $n=2$  is trivial.  If $n=3$, then we can suppose  $\omega_1\ge 2$ by symmetry. Thus the following equalities hold: $I:x_2=(x_1^{\omega_1}x_2^{\omega_1-1},x_3)$ and $(I,x_2)=(x_2)$. It follows from Lemma \ref{quotient} that
$\depth(S/(I:x_2))=1$ and $\depth(S/(I,x_2))=2$. Additionally, $\reg(S/(I:x_2))=2\omega_1-2$ and  $\reg(S/(I,x_2))=0$. Applying  Lemma \ref{exact} to the following short exact sequence
\begin{gather*}
\hspace{3cm}\begin{matrix}
 0 & \rightarrow & \frac{S}{I : x_{2}}(-1)  & \stackrel{ \cdot x_2} \longrightarrow  & \frac{S}{I} & \rightarrow & \frac{S}{(I,x_{2})} & \rightarrow & 0,& \hspace{3cm}(1)
 \end{matrix}
\end{gather*}
we can determine that $\depth (S/I)=1$ and $\reg (S/I)=2\omega_1-1$.

If $n=4$, then  $I=((x_1x_2)^{\omega_1},(x_2x_3)^{\omega_2},(x_3x_4)^{\omega_3})$ with $\omega_2=1$ or $\omega_1=\omega_3=1$.
There are two cases to consider:

(a) Suppose $\omega_2=1$. By symmetry, it can be assumed that $\omega_1\ge \omega_3$. Therefore, we have $(I : x_2)=(x_1^{\omega_1}x_2^{\omega_1-1},x_3)$ and  $(I,x_2)=(x_2,(x_3x_4)^{\omega_3})$. As a result, $\depth(S/(I : x_2))=\depth(S/(I,x_2))=2$,  along with $\reg(S/(I : x_2))=2\omega_1-2$ and $\reg(S/(I,x_2))=2\omega_3-1$, are presented.  Using
 Lemma \ref{exact} and  the exact sequence (1),
we can determine that $\depth(S/I)=2$ and $\reg (S/I)=2\omega_1-1$.

(b) If $\omega_1=\omega_3=1$, then $\omega_2\ge 2$. In this case, we have  $(I : x_2)=(x_1,x_2^{\omega_2-1}x_3^{\omega_2},x_3x_4)$ and  $(I,x_2)=(x_2,x_3x_4)$. Thus $\depth(S/(I,x_2))=2$ and $\reg(S/(I,x_2))=1$. If we let $I'=(I : x_2)$, then   $(I' : x_3)=(x_1,(x_2x_3)^{\omega_2-1},x_4)$ and $(I',x_3)=(x_1,x_3)$. Therefore,  $\depth(S/(I' : x_3))=1$, $\depth(S/(I',x_3))=2$, $\reg(S/(I' : x_3))=2\omega_2-3$ and  $\reg(S/(I',x_3))=0$.
By applying Lemma \ref{exact} to the following two short exact sequences
\begin{gather*}
\begin{matrix}
 0 & \rightarrow & \frac{S}{I' : x_3}(-1)  & \stackrel{ \cdot x_3} \longrightarrow  & \frac{S}{I'} & \rightarrow & \frac{S}{(I',x_3)} & \rightarrow & 0,\\
  0 & \rightarrow & \frac{S}{(I : x_2)}(-1) & \stackrel{ \cdot  x_2} \rightarrow & \frac{S}{I} &\rightarrow & \frac{S}{(I,x_2)} & \rightarrow & 0,
 \end{matrix}
\end{gather*}
we can determine that $\depth(S/I)=1$ and $\reg(S/I)=2\omega_2-1$.$\hfill$
\end{proof}

\begin{Theorem}
\label{path} Let $P_{\omega}^n$ be a path  as in Remark \ref{n-path}, where $n\ge 5$. By   symmetry  and Corollary \ref{integ},  we can assume that  $\omega_i\ge \omega_{i+2}$ and $\omega_i\ge 2$  for some $i\in [n-3]$. Then
\begin{align*}
&\reg(S/I(P_{\omega}^n))=\max \{2\omega_i+\lfloor \frac{i-1}{3} \rfloor+\lfloor \frac{n-(i+1)}{3} \rfloor, 2\omega_{i+2}+\lfloor \frac{i-2}{3} \rfloor+\lfloor \frac{n-i}{3} \rfloor\}-1,\\
&\depth(S/I(P_{\omega}^n))=\min \{\lceil\frac{i}{3}\rceil+\lceil\frac{n-i-a}{3}\rceil, \lceil \frac{i-2}{3} \rceil+\lceil \frac{n-i-2}{3} \rceil+1\},
\end{align*}		
where $a=\begin{cases}
            0, & \text{if $\omega_{i+2}=1$,}        \\
            1, & \text{otherwise.}
        \end{cases}$		
\end{Theorem}

\begin{proof}
Let $I=I(P_{\omega}^n)$ and $I^{\calP}$ be its polarization, then $I^{\calP}=J^{\calP}+K^{\calP}$, which is 
Betti splitting by Lemma \ref{spliting}, where $J=(x_ix_{i+1})^{\omega_i}$ and  $\mathcal{G}(K)=\mathcal{G}(I)\setminus \mathcal{G}(J)$, since $\omega_i\ge 2$ and $\omega_{i-1}=1$. Note that $(J\cap K)^{\calP}=J^{\calP}\cap K^{\calP}$. It follows from Corollary \ref{cor1} and Lemmas
\ref{quotient}  and \ref{polar} that
\begin{eqnarray*}
\reg(S/I)\!\!&=&\!\!\reg(I)-1=\reg(I^{\calP})-1\\
\!\!&=&\!\!\max\,\{\reg(J^{\calP}),\reg(K^{\calP}),\reg(J^{\calP}\cap K^{\calP})-1\}-1\\
\!\!&=&\!\!\max\,\{\reg(J),\reg(K),\reg(J\cap K)-1\}-1 \hspace{4cm} (2)
  \end{eqnarray*}
and
\begin{eqnarray*}
\pd(S/I)\!\!&=&\!\!\pd(I)+1=\pd(I^{\calP})+1\\
\!\!&=&\!\!\max\,\{\pd(J^{\calP}),\pd(K^{\calP}),\pd(J^{\calP}\cap K^{\calP})+1\}+1\\
\!\!&=&\!\!\max\,\{\pd(J),\pd(K),\pd(J\cap K)+1\}+1\\
\!\!&=&\!\!\max\,\{\pd(S/J),\pd(S/K),\pd(S/(J\cap K))+1\}.
  \end{eqnarray*}
Therefore, 
\begin{eqnarray*}
	\depth (S/I)\!\!&=&\!\!n-\pd(S/I)=n-\max\,\{\pd(S/J),\pd(S/K),\pd(S/(J\cap K))+1\}\\
	\!\!&=&\!\!\min\,\{\depth(S/J),\depth(S/K),\depth(S/(J\cap K))-1\}. \hspace{1.5cm} (3)
\end{eqnarray*}
Now, we calculate the depth of $S/I$.

Note that $K=(x_1x_2,\ldots,x_{i-1}x_i)+K'$ where $K'=(x_{i+1}x_{i+2},
 x_{i+2}^{\omega_{i+2}}x_{i+3}^{\omega_{i+2}},\ldots,x_{n-1}x_n)$, 
and $J \cap K=JL$ with $L=I(P_\omega^n \setminus \{x_{i-1},x_i,x_{i+1},x_{i+2}\})+(x_{i-1},x_{i+2})$, where $x_i=0$ if $i \le 0$.
We have the following two cases:

(a) If $\omega_{i+2}=1$, then $\depth (S/J)=n-1$, $\depth (S/K)=\lceil \frac{i}{3} \rceil+\lceil \frac{n-i}{3} \rceil$  and $\depth(S/J\cap K)=\lceil \frac{i-2}{3} \rceil+\lceil \frac{n-i-2}{3} \rceil+2$ by Lemma \ref{trivialpath} and Lemma \ref{sum2}.
It follows from eq. (3) that
\[
\hspace{0.5cm}\depth(S/I)=\min \{\lceil \frac{i}{3} \rceil+\lceil \frac{n-i}{3} \rceil,\lceil \frac{i-2}{3} \rceil+\lceil \frac{n-i-2}{3} \rceil+1\}. \hspace{2.7cm} (4)
\]

(b) If $\omega_{i+2}>1$, then $\depth (S/J)=n-1$, $\depth(S/J\cap K)=\lceil \frac{i-2}{3} \rceil+\lceil \frac{n-i-2}{3} \rceil+2$ by Lemma \ref{trivialpath} and Lemma \ref{sum2}.  By substituting $K'$ for $I$ in equation (4), we can see that $\depth (S/K)=\lceil \frac{i}{3} \rceil+\lceil \frac{n-i-1}{3} \rceil$. Thus
\[
\hspace{0.5cm}\depth(S/I)=\min \{\lceil \frac{i}{3} \rceil+\lceil \frac{n-i-1}{3} \rceil,\lceil \frac{i-2}{3} \rceil+\lceil \frac{n-i-2}{3} \rceil+1\}
\]
by eq. (3).

Next, we calculate the regularity of $S/I$. There are two cases to consider:

 (a) If $\omega_{i+2}=1$, then $\reg(S/J)=2\omega_i-1$, $\reg(S/K)=\lfloor \frac{i+1}{3} \rfloor+\lfloor \frac{n-i+1}{3} \rfloor$ and $\reg(S/J \cap K)=2\omega_i+\lfloor \frac{i-1}{3} \rfloor+\lfloor \frac{n-i-1}{3} \rfloor$ by applying Lemma \ref{trivialpath} and Lemma \ref{sum2}.  From eq. (2), it  follows that
\[
\hspace{3cm}\reg(S/I)=2\omega_i+\lfloor \frac{i-1}{3} \rfloor+\lfloor \frac{n-(i+1)}{3} \rfloor-1. \hspace{3cm}(5)
\]

(b) If $\omega_{i+2}>1$, then $\reg(S/J)=2\omega_i-1$, $\reg(S/J \cap K)=2\omega_i+\lfloor \frac{i-1}{3} \rfloor+\lfloor \frac{n-i-1}{3} \rfloor$ by applying Lemma \ref{trivialpath} and Lemma \ref{sum2}, and $\reg(S/K)=2\omega_{i+2}+\lfloor \frac{i+1}{3} \rfloor+\lfloor \frac{n-i}{3} \rfloor-2$ by substituting $K'$ for $I$ in equation (5). Therefore,
\[
\hspace{0.5cm}\reg(S/I(P_{\omega}^n))=\max \{2\omega_i+\lfloor \frac{i-1}{3} \rfloor+\lfloor \frac{n-(i+1)}{3} \rfloor, 2\omega_{i+2}+\lfloor \frac{i-2}{3} \rfloor+\lfloor \frac{n-i}{3} \rfloor\}-1
\]
by eq. (2).$\hfill$
\end{proof}

We will now  analyze the powers of the edge ideal of an edge-weighted  path $P_{\omega}^n$.
\begin{Lemma}\label{pathcolon}
Let $P^n_\omega$ be an edge-weighted path with the vertex set   $\{x_1,\ldots,x_n\}$ and the edge set  $\{e_1,\ldots,e_{n-1}\}$, where $e_i=\{x_i,x_{i+1}\}$ and $\omega_i=\omega(e_i)$ for all $i \in [n-1]$. If $\omega_{n-1}=1$, then for all $t\ge 2$,  we have
\begin{itemize}
\item[(1)] $(I(P^n_\omega)^t: x_{n-1}x_{n})=I(P^n_\omega)^{t-1}$;
 \item[(2)] $((I(P^n_\omega)^t: x_n),x_{n-1})=(I(P^n_\omega \setminus {x_{n-1}})^t,x_{n-1})$;
 \item[(3)] $(I(P^n_\omega)^t,x_n)=(I(P^n_\omega\setminus {x_n})^t,x_n)$;
 \item[(4)] $(I(P^n_\omega)^t,x_{n-1})=(I(P^n_\omega\setminus {x_{n-1}})^t,x_{n-1})$;
\item[(5)] $((I(P^n_\omega)^t : x_{n-1}),x_n)=((I(P^n_\omega\setminus {x_n})^t : x_{n-1}),x_n)$.
 \end{itemize}
  \end{Lemma}

\begin{proof}Let $I=I(P^n_\omega)$.

(1) For any monomial $u \in \mathcal{G}(I^t: x_{n-1}x_{n})$, it follows that  $ux_{n-1}x_n \in I^t$. If
\[
\hspace{3cm}ux_{n-1}x_n=u_{i1} \cdots u_{it}h \hspace{5.7cm}(6)
\]
  where each $u_{ij} \in \mathcal{G}(I)$  and $h$ is a  monomial,
and $x_n| u_{ij}$ for some $j\in [t]$, then $u_{ij}=x_{n-1}x_n$, since $N_G(x_n)=\{x_{n-1}\}$. Therefore,
$u\in I^{t-1}$ by eq.  (6). If $x_n\nmid u_{ij}$  for any $j\in [t]$, then $x_n|h$ according to   eq.  (6). This implies that
$u\in I^{t-1}$, since $\omega_{n-1}=1$.

(2) It is clear that $(I(P^n_\omega \setminus {x_{n-1}})^t,x_{n-1})\subseteq ((I^t: x_n),x_{n-1})$.  For any monomial $u \in \mathcal{G}(((I^t: x_n),x_{n-1}))$. If
 $x_{n-1}|u$, then  $u \in (I(P^n_\omega \setminus {x_{n-1}})^t,x_{n-1})$. Otherwise,   $ux_n \in I^t$. We can write $ux_n$ as 
\[
ux_n=u_{\ell 1} \cdots u_{\ell t}v
\]
  where each $u_{\ell j} \in \mathcal{G}(I(P^n_\omega\setminus {x_{n-1}}))$  and $v$ is a  monomial. It follows that $x_n|v$ since  $x_{n-1}
  \nmid u$. Therefore,  $u \in I(P^n_\omega \setminus {x_{n-1}})^t$.

(3) For any monomial $u \in \mathcal{G}(I^t) \setminus \mathcal{G}(I(P^n_\omega\setminus {x_n})^t)$, it follows that $x_n|u$. So $u \in (x_n)$.

(4)  If $u$ is a monomial in $\mathcal{G}(I^t) \setminus \mathcal{G}(I(P^n_\omega\setminus {x_{n-1}})^t)$, then $x_{n-1}|u$, indicating that $u \in (x_{n-1})$.

(5)  For any monomial $u \in \mathcal{G}((I^t : x_{n-1}),x_n)$, if
 $x_{n}|u$, then  $u \in ((I(P^n_\omega \setminus {x_{n}})^t : x_{n-1}),x_{n})$. Otherwise, we have $ux_{n-1} \in I^t$. Let $ux_{n-1}=u_{i1} \ldots u_{it}g$, where each
 $u_{ij} \in \mathcal{G}(I)\setminus \{x_{n-1}x_n\}$ and $g$ is a monomial. This implies that each $u_{ij} \in \mathcal{G}(I(P^n_\omega\setminus {x_{n}}))$. Hence $u \in I(P^n_\omega \setminus {x_{n}})^t : x_{n-1}$.
 $\hfill$
\end{proof}

\begin{Lemma}\label{pathcolon2}
 Let $P_{\omega}^4$ be a path as in Remark \ref{n-path}. 
Suppose $\omega_1 \ge \omega_3 \ge 2$ and $\omega_2=1$. Then, for any $t\ge 2$, we have
\begin{itemize}
\item[(1)] $(I(P^4_\omega)^t:(x_2x_3)^{t-1})=I(P^4_\omega)$;
\item[(2)] $((I(P^4_\omega)^t:(x_2x_3)^\ell),x_2x_3)=((x_1x_2)^{(t-\ell)\omega_1},x_2x_3,(x_3x_4)^{(t-\ell)\omega_3})$ for  $\ell\in [t-2]$;
\item[(3)] $(I(P^4_\omega)^t,x_2x_3)=((x_1x_2)^{t\omega_1},x_2x_3,(x_3x_4)^{t\omega_3})$.
 \end{itemize}
\end{Lemma}

\begin{proof}Let $I=I(P^4_\omega)$.

(1) It is evident that $I\subseteq (I^t \colon (x_2x_3)^{t-1})$.  For any monomial $u$  in $\mathcal{G}(I^t \colon (x_2x_3)^{t-1})$,
if $x_2x_3|u$, then  $u\in I$. Otherwise, we write  $u(x_2x_3)^{t-1}$ as $u(x_2x_3)^{t-1}=(x_2x_3)^{t-j}h\prod_{\ell=1}^{j}u_{i\ell}$ for some $j\in [t]$, where each
$u_{i\ell} \in \mathcal{G}(I)$,  $x_2x_3\nmid h$ and  $x_2x_3\nmid u_{i\ell}$. Thus
$u(x_2x_3)^{j-1}=h\prod_{\ell=1}^{j}u_{i\ell}$. Suppose that $u_{i\ell}=(x_1x_2)^{\omega_1}$ for any $\ell\in [k]$ and  $u_{i\ell}=(x_3x_4)^{\omega_3}$ for $k+1\le \ell\le j$, then
\[
\hspace{3cm}u(x_2x_3)^{j-1}=(x_1x_2)^{k\omega_1}(x_3x_4)^{(j-k)\omega_3}h \hspace{3.7cm}(7)
\]
We can distinguish six cases as follows:

(i) If $j=2k+1$, then $u(x_2x_3)^{2k}=(x_1x_2)^{k\omega_1}(x_3x_4)^{(k+1)\omega_3}h$ by eq. (7).  As $\omega_3\ge 2$, $\omega_3(k+1)\ge \omega_3+2k$.
  Therefore, $(x_3x_4)^{\omega_3}|u$.

 (ii) If $j>2k+1$ and is odd, then $2(j-k-1)>j-1$. Hence $\omega_3(j-k)>(j-1)+\omega_3$, since $\omega_3\ge 2$.
 By comparing the powers of $x_3$ in  eq.  (7), we can deduce that $(x_3x_4)^{\omega_3}|u$.

 (iii)  If $j<2k+1$ and is odd, then $j-1\le 2(k-1)$.  Thus, we have  $k\omega_1\ge 2(k-1)+\omega_1\ge  (j-1)+\omega_1$. Therefore, we can conclude that  $(x_1x_2)^{\omega_1}|u$ by comparing powers of $x_2$ in eq.  (7).

 (iv) If $j=2k$, then $\omega_1k, \omega_3(j-k) \ge 2k$. Hence, $x_2x_3|u$ by eq. (7).

 (v) If $j>2k$ and is even, then $2(j-k-1)\ge j-1$. It follows that $\omega_3(j-k) \ge 2(j-k-1)+\omega_3 \ge (j-1)+\omega_3$.  As a result,  $(x_3x_4)^{\omega_3}|u$ according to  eq.  (7).

 (vi) If $j<2k$ and  is even,  then $k \omega_1=(k-1)\omega_1+\omega_1 \ge 2(k-1)+\omega_1 \ge (j-1)+\omega_1$. Hence, $(x_1x_2)^{\omega_1}|u$ by using eq.   (7).

In any case, we always have  $u \in I$.

(2) Let $u \in \mathcal{G}(I^t:(x_2x_3)^\ell)$, and $x_2x_3|u$, then  $u \in ((x_1x_2)^{(t-\ell)\omega_1},
x_2x_3,(x_3x_4)^{(t-\ell)\omega_3})$. Otherwise, we write $u(x_2x_3)^\ell$ as 
$u(x_2x_3)^\ell=(x_1x_2)^{i\omega_1}(x_2x_3)^j(x_3x_4)^{k\omega_3}h$, where $h$ is a monomial such that $x_2x_3\nmid h$ and $i+j+k=t$ with $j \le \ell$,
then
\[
\hspace{3.0cm}u(x_2x_3)^{\ell-j}=(x_1x_2)^{i\omega_1}(x_3x_4)^{k\omega_3}h.\hspace{3.9cm}(8)
\]
We distinguish between the following two cases:

(i) When $i\le k$,  we can rewrite eq.  (8) as
\[
\hspace{3.0cm} u(x_2x_3)^{\ell-j}=(x_2x_3)^{i\omega_3}(x_3x_4)^{(k-i)\omega_3}h'\hspace{3.7cm}(9)
\]
in which $h'=x_1^{i\omega_1}x_2^{i(\omega_1-\omega_3)}x_4^{i\omega_3}h$,   based on the power of $x_3$.
Since $x_2x_3\nmid h$, we have $\ell-j \ge i\omega_3$ by comparing  powers of $x_2x_3$ in equation  (9). Therefore,  $u(x_2x_3)^{\ell-(j+i\omega_3)}=(x_3x_4)^{(k-i)\omega_3}h'$. Note that $\omega_1 \ge \omega_3\ge 2$ and  $i+j+k=t$ with $\ell-j \ge i\omega_3$, we have
\begin{eqnarray*}
& &(k-i)\omega_3+i\omega_3-(\ell-j)=k\omega_3+j-\ell=(t-i-j)\omega_3+j-\ell+\ell\omega_3-\ell\omega_3\\
&=&(t-\ell)\omega_3-i\omega_3+(\ell-j)(\omega_3-1)\ge (t-\ell)\omega_3-i\omega_3+(\ell-j)\\
& \ge &(t-\ell)\omega_3-i\omega_3+i\omega_3=(t-\ell)\omega_3
\end{eqnarray*}
and
\[
k-i=k+\ell-i-(i+j+k)+(t-\ell)=(\ell-j-2i)+(t-\ell)\ge (i\omega_3-2i)+(t-\ell) \ge (t-\ell).
\]
 By comparing the powers of $x_3x_4$ in  eq. (9), we obtain  $(x_3x_4)^{(t-\ell)\omega_3}|u$, which implies
 $u \in ((x_1x_2)^{(t-\ell)\omega_1},x_2x_3,(x_3x_4)^{(t-\ell)\omega_3})$.

\item[(ii)] When $i>k$,  equation  (8) can be rewritten as
\[
\hspace{3.0cm}u(x_2x_3)^{\ell-j}=(x_2x_3)^{k\omega_3}(x_1x_2)^{(i-k)\omega_1}v\hspace{3.7cm}(10)
 \]
with $v=x_1^{k\omega_1}x_2^{k(\omega_1-\omega_3)}x_4^{k\omega_3}h$.
Since $(x_2x_3)\nmid u$, one has $\ell-j \ge k\omega_3$ by comparing  powers of $x_2x_3$ in eq. (10).
Thus,  $u(x_2x_3)^{\ell-j-\omega_3k}=(x_1x_2)^{(i-k)\omega_1}v$. Note that $\omega_1 \ge \omega_3\ge 2$ and  $i+j+k=t$ with $\ell-j \ge k\omega_3$, we have
\begin{eqnarray*}
& &(i-k)\omega_1+k\omega_3-(\ell-j)=(t-j-k)\omega_1+k(\omega_3-\omega_1)+\ell\omega_1-\ell\omega_1-(\ell-j)\\
&=&(t-\ell)\omega_1+k(\omega_3-2\omega_1)+(\ell-j)(\omega_1-1)\\
& \ge &(t-\ell)\omega_1+k(\omega_3-2\omega_1)+k\omega_3(\omega_1-1)=(t-\ell)\omega_1+k\omega_1(\omega_3-2)\\& \ge & (t-\ell)\omega_1
\end{eqnarray*}
and
\[
i-k=(t-j-k)-k+\ell-\ell=(t-\ell)-2k+(\ell-j) \ge (t-\ell)-2k+k\omega_3 \ge (t-\ell).
\]
 By comparing powers of $x_1x_2$ in eq.  (10), we obtain $(x_1x_2)^{(t-\ell)\omega_1}|u$, which implies  that 
 $u \in ((x_1x_2)^{(t-\ell)\omega_1},x_2x_3,(x_3x_4)^{(t-\ell)\omega_3})$.

 (3) Let $u \in \mathcal{G}(I^t)$, and $x_2x_3|u$, then   $u \in ((x_1x_2)^{t\omega_1},x_2x_3,(x_3x_4)^{t\omega_3})$. Otherwise, let
$u=(x_1x_2)^{i\omega_1}(x_3x_4)^{j\omega_3}$, where $i+j=t$ and $j \le \ell$.  If $i,j\neq 0$, then  $x_2x_3|u$. If $i=0$, then $j=t$, which means that
$(x_3x_4)^{t\omega_3}|u$. If $j=0$, then $i=t$, which results in
$(x_1x_2)^{t\omega_3}|u$. In any case, we always have $u \in ((x_1x_2)^{t\omega_1},x_2x_3,(x_3x_4)^{t\omega_3})$.$\hfill$
 \end{proof}

Using the previous preparatory knowledge, we can calculate the regularity and depth of the edge ideal of powers of a path $P_{\omega}^n$, starting with  $n\le 4$.
\begin{Theorem} \label{small2}
Let $P_{\omega}^n$ be a path as in Remark \ref{n-path}, where $n\le 4$.  Let $\omega=\max\{\omega_1,\ldots,\omega_{n-1}\}$. 
Then, for all $t\ge 1$, we have the following results:

\begin{itemize}
\item[(1)] If $n=2$, then $\reg(S/I(P_{\omega}^n)^t) = 2t\omega-1$ and  $\depth(S/I(P_{\omega}^n)^t) = 1$.
\item[(2)] If $n=3$, then $\reg(S/I(P_{\omega}^n)^t) = 2t\omega-1$ and   $\depth(S/I(P_{\omega}^n)^t) = 1$.
\item[(3)] If $n=4$, then $\reg(S/I(P_{\omega}^n)^t) = 2t\omega-1$ and \\
$\depth(S/I(P_{\omega}^n)^t)=\begin{cases}
            1, & \text{if $\omega_1=\omega_3=1$ and $\omega_2 > 1$,}        \\
            2, & \text{if $\omega_1,\omega_3 > 1$ and $\omega_2=1$.}
        \end{cases}$
\end{itemize}
\end{Theorem}

\begin{proof}
Let $I=I(P_{\omega}^n)$. The case where $n=2$ is straightforward. If $n=3$, then $I=((x_1x_2)^{\omega_1},x_2x_3)$ with $\omega_1\ge 2$,  as proven in  Lemma  \ref{integral}. We will prove these assertions  by induction on $t$. The case where $t=1$ follows from Theorem \ref{smalln}. In the following,  we  assume that $t \ge 2$. Consider the following short exact sequences
\begin{gather*}
\hspace{2.0cm}\begin{matrix}
 0 & \rightarrow & \frac{S}{I^t : x_3}(-1)  & \stackrel{ \cdot x_3} \longrightarrow  & \frac{S}{I^t} & \rightarrow & \frac{S}{(I^t, x_3)} & \rightarrow & 0 ,& \hspace{1.2cm}(11)\\
  0 & \rightarrow & \frac{S}{(I^t : x_3) : x_2}(-1) & \stackrel{ \cdot  x_2} \longrightarrow & \frac{S}{I^t : x_3} &\rightarrow & \frac{S}{((I^t : x_3), x_2)} & \rightarrow & 0.
 \end{matrix}
\end{gather*}
Using  Lemma \ref{sum2}, Lemma \ref{pathcolon} and the inductive hypothesis, we can deduce that $\reg(S/((I^t : x_3): x_2))=\reg(S/I^{t-1})=2(t-1)\omega_1-1$,
  $\reg(S/((I^t: x_3),x_2)=\reg(S/(x_2))=0$ and $\reg(S/(I^t,x_3))=\reg(S/((x_1x_2)^{t\omega_1},x_3))=2t\omega_1-1$. Additionally, we can see that  $\depth(S/((I^t : x_3): x_2))=\depth(S/I^{t-1})=1$, $\depth(S/((I^t: x_3),x_2)=\depth(S/(x_2))=2$ and $\depth(S/(I^t,x_3))=\depth(S/((x_1x_2)^{t\omega_1},x_3))=1$.
Using Lemma \ref{exact} and  the short exact sequences  (11), we can determine that $\reg (S/I^t)=2t\omega_1-1$ and $\depth(S/I^t)=1$.

If $n=4$, then $I=((x_1x_2)^{\omega_1},(x_2x_3)^{\omega_2},(x_3x_4)^{\omega_3})$ with $\omega_2=1$ or $\omega_1=\omega_3=1$ by  Lemma  \ref{integral}.
There are two cases to consider:

Firstly, if $\omega_2=1$, then we can assume that $\omega_1\ge \omega_3$ by symmetry. The assertions are proven by induction on $t$ with the case $t=1$ verified  in Theorem \ref{smalln}. We assume that $t \ge 2$ 
and proceed as follows: 

(a) If $\omega_3=1$, then  $(I^t,x_3x_4)=(J^t,x_3x_4)$ where  $J=((x_1x_2)^{\omega_1},x_2x_3)$. Consider the following short exact sequences:
\begin{gather*}
\begin{matrix}
 0 & \rightarrow & \frac{S}{I^t:x_3x_4}(-2)  & \stackrel{\cdot x_3x_4} \longrightarrow  & \frac{S}{I^t} & \rightarrow & \frac{S}{(I^t, x_3x_4)} & \rightarrow & 0,  \\
  0 & \rightarrow & \frac{S}{(J^t, x_3x_4): x_3}(-1) & \stackrel{ \cdot x_3} \longrightarrow & \frac{S}{(J^t, x_3x_4)} &\rightarrow & \frac{S}{(J^t,x_3)} & \rightarrow & 0, &  \hspace{0.8cm}(12)\\
 0&  \rightarrow & \frac{S}{((J^t, x_3x_4): x_3): x_2}(-1) & \stackrel{ \cdot x_2} \longrightarrow & \frac{S}{(J^t, x_3x_4): x_3)}& \rightarrow & \frac{S}{(((J^t, x_3x_4): x_3),x_2)}& \rightarrow & 0.
 \end{matrix}
\end{gather*}
Lemma \ref{pathcolon} implies that  $(I^t:x_3x_4)=I^{t-1}$, $(J^t,x_3)=((x_1x_2)^{t\omega_1},x_3)$, $((J^t, x_3x_4): x_3) : x_2=(J^{t-1},x_4)$, $(((J^t, x_3x_4) : x_3),x_2)=(x_2,x_4)$. Therefore,  $\reg(S/(J^{t-1},x_4))=2(t-1)\omega_1-1$, $\reg(S/I^{t-1})=2(t-1)\omega_1-1$, $\reg (S/((x_1x_2)^{t\omega_1},x_3))=2t\omega_1-1$ and  $\reg (S/(x_2,x_4))=0$   by Lemma \ref{sum2} and
the inductive hypothesis.
From Lemma \ref{exact} and the above  short exact sequences (12), we can conclude $\reg (S/I^t)=2t\omega_1-1$.

(b) If $\omega_3>1$, then  $I=((x_1x_2)^{\omega_1},x_2x_3,(x_3x_4)^{\omega_3})$. Consider
the  exact sequences:
\begin{gather*}
\hspace{1.0cm}\begin{matrix}
 0 & \rightarrow & \frac{S}{I^t: x_2x_3}(-2)  & \stackrel{ \cdot x_2x_3} \longrightarrow  & \frac{S}{I^t} & \rightarrow & \frac{S}{(I^t, x_2x_3)} & \rightarrow & 0, \\
  0 & \rightarrow & \frac{S}{I^t:(x_2x_3)^2}(-2) & \stackrel{ \cdot  x_2x_3} \longrightarrow & \frac{S}{I^t \colon x_2x_3} &\rightarrow & \frac{S}{((I^t:x_2x_3), x_2x_3)} & \rightarrow & 0,  \\
     &  &\vdots&  &\vdots&  &\vdots&  &\\
 0&  \rightarrow & \frac{S}{I^t: (x_2x_3)^{t-2}}(-2) & \stackrel{ \cdot x_2x_3} \longrightarrow & \frac{S}{I^t: (x_2x_3)^{t-3}}& \rightarrow & \frac{S}{((I^t: (x_2x_3)^{t-3}), x_2x_3)}& \rightarrow & 0, \\
 0&  \rightarrow & \frac{S}{I^t : (x_2x_3)^{t-1}}(-2) & \stackrel{ \cdot x_2x_3} \longrightarrow & \frac{S}{I^t: (x_2x_3)^{t-2}}& \rightarrow & \frac{S}{((I^t: (x_2x_3)^{t-2}), x_2x_3)}& \rightarrow & 0.
 \end{matrix}
\end{gather*}
Since $I^t : (x_2x_3)^{t-1}=I$, $((I^t :(x_2x_3)^{i}), x_2x_3)=((x_1x_2)^{(t-i)\omega_1},x_2x_3,(x_3x_4)^{(t-i)\omega_3})$ for any $i\in [t-2]$, and
 $(I^t, x_2x_3)=((x_1x_2)^{t\omega_1},x_2x_3,(x_3x_4)^{t\omega_3})$  by Lemma \ref{pathcolon2}. By Theorem \ref{smalln}, we can deduce that 
 $\reg(S/((I^t :(x_2x_3)^{i}), x_2x_3))=2(t-i)\omega_1-1$ and  $\depth(S/((I^t :(x_2x_3)^{i}), x_2x_3))=2$ for any $i=0,\ldots,t-1$. From Lemma \ref{exact} and  the above exact sequences, we can conclude that $\reg (S/I^t)=2t\omega_1-1$ and $\depth(S/I^t)=2$. 

Next,  let $\omega_1=\omega_3=1$, then $\omega_2\ge 2$, since $P_{\omega}^4$  has non-trivial weight. We will prove the assertions by induction on $t$ with
the case $t=1$  verified  in Theorem \ref{smalln}. Now we assume that $t \ge 2$ and we consider the  exact sequences
  \begin{gather*}
\hspace{2cm}\begin{matrix}
 0 & \longrightarrow & \frac{S}{I^t : x_4}(-1)  & \stackrel{ \cdot x_4} \longrightarrow  & \frac{S}{I^t} & \longrightarrow & \frac{S}{(I^t, x_4)} & \longrightarrow & 0, & \hspace{1cm}(13)\\
  0 & \longrightarrow & \frac{S}{(I^t : x_4) : x_3}(-1) & \stackrel{ \cdot  x_3} \longrightarrow & \frac{S}{I^t : x_4} &\longrightarrow & \frac{S}{((I^t : x_4),x_3)} & \longrightarrow & 0.
 \end{matrix}
\end{gather*}
Let $J=(x_1x_2,(x_2x_3)^{\omega_2})$, then, using Lemma \ref{pathcolon},  Lemma \ref{sum2} and the inductive hypothesis, we can determine that
$\reg(S/((I^t:x_4):x_3))=\reg(S/I^{t-1})=2(t-1)\omega_2-1$, $\reg(S/((I^t: x_4),x_3))=\reg(S/((x_1x_2)^t,x_3))=2t-1$ and $\reg(S/(I^t,x_4))=\reg(S/(J^t,x_4))=2t\omega_2-1$. Additionally, we can also see that $\depth(S/((I^t:x_4):x_3))=\depth(S/I^{t-1})=1$, $\depth(S/((I^t: x_4),x_3)=\depth(S/((x_1x_2)^t,x_3))=2$ and $\depth(S/(I^t,x_4))=\depth(S/(J^t,x_4))=1$. 
Applying Lemma \ref{exact} to the  exact sequences (13), we obtain that
 $\reg (S/I^t)=2t\omega_2-1$ and $\depth(S/I^t)=1$.
This concludes the proof.$\hfill$
\end{proof}

 The following proposition can be shown  using similar arguments as the proof of Theorem \ref{small2}, provided that $\omega_2\ge 2$ and $\omega_1=\omega_3=1$,  we omit its proof.
\begin{Proposition}\label{path power5}
Let $P_\omega^4$ be a  path  as in Remark \ref{n-path}. If $\omega_1 > 0$ and $\omega_2=\omega_3=1$, then $\depth(S/I(P_{\omega}^4)^t) \ge 1$ for any $t \ge 2$.
\end{Proposition}

In the following, we compute the  regularity of powers of the edge ideal of a path $P_{\omega}^n$ with $n\ge 5$, as described in Remark \ref{n-path}.

\begin{Theorem}
\label{path power3} 
 Let $P_{\omega}^n$ be a path as  in Remark \ref{n-path}, where $n\ge 5$. 
If $\omega_1=\max\{\omega_i\mid i\in [n-1]\}$, then  $\reg(S/I(P_{\omega}^n)^t) \le 2(t-1)\omega_1+\reg(S/I(P_{\omega}^n))$.
\end{Theorem}

\begin{proof}Let $I=I(P_{\omega}^n)$. We  will prove the assertions  by induction on $n$ and $t$. The case where $t=1$ is trivial. In the following,  we assume that $t \ge 2$.
We have two cases:

(a) If  $n=5$, then, by Theorems \ref{smalln},  \ref{path},  Lemma \ref{pathcolon} and the inductive hypothesis, we can derive that $\reg(S/((I^t: x_5):x_{4})=\reg(S/I^{t-1})\le 2\omega_1+2(t-2)\omega_1$, $\reg(S/((I^t:x_5),x_{4})=\reg(S/(I(P_\omega^5\setminus {x_{4}})^t,x_{4})\le (2\omega_1-1)+2(t-1)\omega_1$ and $\reg(S/(I^t,x_5))=\reg(S/(I(P_\omega^5\setminus{x_5})^t,x_5))\le (2\omega_1-1)+2(t-1)\omega_1$. 
The desired result follows  from Theorem   \ref{path} and the following  exact sequences
\begin{gather*}
\hspace{2cm}\begin{matrix}
 0 & \rightarrow & \frac{S}{I^t : x_5}(-1)  & \stackrel{ \cdot x_5} \longrightarrow  & \frac{S}{I^t} & \rightarrow & \frac{S}{(I^t, x_5)} & \rightarrow & 0, \\
  0 & \rightarrow & \frac{S}{(I^t : x_5) : x_4}(-1) & \stackrel{ \cdot  x_4} \longrightarrow & \frac{S}{I^t : x_5} &\rightarrow & \frac{S}{((I^t : x_5),x_4)} & \rightarrow & 0.
 \end{matrix}
\end{gather*}

(b) Suppose that  $n \ge 6$. In this case, let $J=I(P_\omega^n\setminus{x_n})$, then,   by Lemma \ref{pathcolon}, we have  $(J^t: x_{n-1}):x_{n-2}=J^{t-1}$, $((J^t: x_{n-1}),x_{n-2})=(I(P_\omega^n\setminus {\{x_n,x_{n-2}\}})^t,x_{n-2})$, $(I^t: x_{n-1}): x_n=I^{t-1}$, $((I^t: x_{n-1}),x_n)=((I(P_\omega^n\setminus{x_n})^t: x_{n-1}),x_n)$ and $(I^t,x_{n-1})=(I(P_\omega^n\setminus{x_{n-1}})^t,x_{n-1})$.
It follows from Theorem \ref{path} and the inductive hypothesis that
\begin{eqnarray*}
\reg(S/((J^t: x_{n-1}):x_{n-2}))\!\!\!& \le &\!\!\!(2\omega_1-1)+\lfloor \frac{n-3}{3} \rfloor+2(t-2)\omega_1,\\
\reg(S/((J^t: x_{n-1}),x_{n-2}))\!\!\!& \le &\!\!\!(2\omega_1-1)+\lfloor \frac{n-5}{3} \rfloor+2(t-1)\omega_1,\\
\reg(S/((I^t: x_{n-1}): x_n))\!\!\!& \le &\!\!\!(2\omega_1-1)+\lfloor \frac{n-2}{3} \rfloor+2(t-2)\omega_1,\\
\reg(S/(I^t,x_{n-1}))\!\!\!& \le &\!\!\!(2\omega_1-1)+\lfloor \frac{n-4}{3} \rfloor+2(t-1)\omega_1.
\end{eqnarray*}
Applying Lemma \ref{exact} again to the following exact sequences
\begin{gather*}
\hspace{1cm}\begin{matrix}
 0 & \rightarrow & \frac{S}{I^t : x_{n-1}}(-1)  & \stackrel{ \cdot x_{n-1}} \longrightarrow  & \frac{S}{I^t} & \rightarrow & \frac{S}{(I^t, x_{n-1})} & \rightarrow & 0, \\
  0 & \rightarrow & \frac{S}{(I^t : x_{n-1}) : x_n}(-1) & \stackrel{ \cdot  x_n} \longrightarrow & \frac{S}{I^t : x_{n-1}} &\rightarrow & \frac{S}{((I^t : x_{n-1}),x_n)} & \rightarrow & 0, \\
 0 & \rightarrow & \frac{S}{(J^t : x_{n-1}) : x_{n-2}}(-1) & \stackrel{ \cdot  x_{n-2}} \longrightarrow & \frac{S}{J^t : x_{n-1}} &\rightarrow & \frac{S}{((J^t : x_{n-1}),x_{n-2})} & \rightarrow & 0,
 \end{matrix}
\end{gather*}
one has $\reg(S/I^t) \le 2\omega_1-1+\lfloor \frac{n-2}{3} \rfloor+2(t-1)\omega_1=\reg(S/I)+2(t-1)\omega_1$.
\end{proof}

\begin{Theorem}
\label{path power4}
Let $P_{\omega}^n$ be a path  as in Remark \ref{n-path}, where  $n\ge 5$. 
Then,  for any $t\ge 1$, we have
\[
\reg(S/I(P_{\omega}^n)^t) \le \reg(S/I(P_{\omega}^n))+2(t-1)\omega
\]
where $\omega=\max\{\omega_i\mid i\in [n-1]\}$.
\end{Theorem}

\begin{proof}
Let $I=I(P_{\omega}^n)$. Since $P_{\omega}^n$ has non-trivial weight, $\omega\ge 2$. If  $\omega=\omega_i$, and  $i=1$, then the desired results follow from Theorem \ref{path power3}. Assume that $i\ge 2$
 and we can prove the statements  by induction on $n$ and $t$.  The  case where $t=1$ is trivial. Now, assuming that  $t \ge 2$.

First, since $P_{\omega}^n$ is  integrally closed, it  has at most two edges with non-trivial weights by Corollary  \ref{integ}. Therefore, for $\omega_i$ and $\omega_{i+2}$, by symmetry, there exist two cases: (a) $\omega_i > \omega_{i+2} \ge 1$, or (b) $\omega_i=\omega_{i+2} \ge 2$, while for other edges $\omega_j$ where $j \neq i,i+2$,  we have $\omega_j=1$. We distinguish into two cases:

(1) If $i=2$, then  $(I^t : x_1) : x_2=I^{t-1}$, $((I^t : x_1),x_{2})=(I(P_\omega^n\setminus {x_{2}})^t,x_{2})$ and $(I^t,x_1)=(I(P_\omega^n\setminus{x_1})^t,x_1)$ by Lemma \ref{pathcolon}. We have two subcases:

(a) If $\omega_{i}>\omega_{i+2} \ge1$, then by the inductive hypothesis and Theorem \ref{path}, we can conclude that
\begin{align*}
\reg(S/((I^t : x_1): x_{2})&\le(2\omega-1)+\lfloor \frac{n-3}{3} \rfloor+2(t-2)\omega,\\
\reg(S/(I^t,x_1))&\le(2\omega-1)+\lfloor \frac{n-3}{3} \rfloor+2(t-1)\omega.\hspace{2.0cm}(14)\\
\end{align*}
We will now prove by induction on  $t$ that $\reg(S/((I^t : x_1),x_{2}))\le  (2\omega_{4}-1)+\lfloor \frac{n-5}{3} \rfloor+2(t-1)\omega_{4}$ holds for all $t\ge 1$. The  case $t=1$ has been  verified in  Theorem \ref{path}.
Now assume that $t\ge 2$.
Let  $L=I(P_\omega^n\setminus {x_{2}})$,  $M=(x_5x_6,\ldots,x_{n-1}x_n)$, then   by Lemma \ref{pathcolon},  we obtain that  $L=M+(x_3x_4,(x_4x_5)^{\omega_4})$, $(L^t : x_{3}) : x_{4}=L^{t-1}$, $((L^t : x_{3}),x_{4})=(M^t,x_{4})$ and $(L^t,x_{3})=(N^t,x_{3})$, where $N=M+((x_4x_5)^{\omega_4})$.
 Therefore, we can deduce from Theorem \ref{path power3} and the inductive hypothesis that
\begin{eqnarray*}
\reg(S/((L^t : x_{3}) : x_{4})&\le&  (2\omega_{4}-1)+\lfloor \frac{n-5}{3} \rfloor+2(t-2)\omega_{4}, \\
\reg(S/((L^t : x_{3}),x_{4}))&\le&  \lfloor \frac{n-3}{3} \rfloor+2(t-1),\\
\reg(S/(L^t,x_{3}))&\le& (2\omega_{4}-1)+\lfloor \frac{n-5}{3} \rfloor+2(t-1)\omega_{4}.
\end{eqnarray*}
Applying Lemma \ref{exact} to the following exact sequences
\begin{gather*}
\hspace{1cm}\begin{matrix}
 0 & \rightarrow & \frac{S}{L^t : x_{3}}(-1)  & \stackrel{ \cdot x_{3}} \longrightarrow  & \frac{S}{L^t} & \rightarrow & \frac{S}{(L^t, x_{3})} & \rightarrow & 0, \\
  0 & \rightarrow & \frac{S}{(L^t : x_{3}) : x_{4}}(-1) & \stackrel{ \cdot  x_{4}} \longrightarrow & \frac{S}{L^t : x_{3}} &\rightarrow & \frac{S}{((L^t : x_{3}),x_{4})} & \rightarrow & 0,
 \end{matrix}
\end{gather*}
 we obtain  $\reg(S/((I^t : x_1),x_{2}))\le  (2\omega_{4}-1)+\lfloor \frac{n-5}{3} \rfloor+2(t-1)\omega_{4}$.

(b)  If  $\omega_i=\omega_{i+2}\ge 2$, then,  by Theorem \ref{path} and the inductive
hypothesis,    we obtain
\begin{eqnarray*}
\reg(S/((I^t : x_1): x_{2}))&\le&2\omega+\lfloor \frac{n-5}{3} \rfloor+2(t-2)\omega,\\
\reg(S/(I^t,x_1))&\le&(2\omega-1)+\lfloor \frac{n-3}{3} \rfloor+2(t-1)\omega.\hspace{2.0cm}(15)\\
\end{eqnarray*}
Using similar arguments as the proof of  regularity of $S/((I^t : x_1),x_{2})$ in part (a), We can also obtain 
\begin{eqnarray*}
\reg(S/((I^t  \colon x_1),x_{2}))& \le &(2\omega-1)+\lfloor \frac{n-5}{3} \rfloor+2(t-1)\omega. \hspace{2.0cm}(16)\\
\end{eqnarray*}

In the two cases mentioned above, we can obtain $\reg(S/I(P_{\omega}^n)^t) \le \reg(S/I(P_{\omega}^n))+2(t-1)\omega$ by applying  formulas (14)$\sim$(16) and Lemma \ref{exact}  to the  following exact sequences
\begin{gather*}
\hspace{1cm}\begin{matrix}
 0 & \rightarrow & \frac{S}{I^t : x_1}(-1)  & \stackrel{ \cdot x_1} \longrightarrow  & \frac{S}{I^t} & \rightarrow & \frac{S}{(I^t, x_1)} & \rightarrow & 0,\\
  0 & \rightarrow & \frac{S}{(I^t : x_1) : x_{2}}(-1) & \stackrel{ \cdot  x_{2}} \rightarrow & \frac{S}{I^t : x_1} &\rightarrow & \frac{S}{((I^t : x_1),x_{2})} & \rightarrow & 0.
 \end{matrix}
\end{gather*}

(2) If $i\ge 3$, then $n \ge 6$.  By Lemma \ref{pathcolon}, we can see that
$(I^t,x_{2})=(I(P_\omega^n\setminus{x_{2}})^t,x_{2})$,  $(I^t : x_{2}): x_1=I^{t-1}$, $((I^t : x_{2}),x_1)=((K^t:x_{2}),x_1)$, $(K^t : x_{2}) : x_{3}=K^{t-1}$ and $((K^t : x_{2}),x_{3})=(I(P_\omega^n\setminus {\{x_1,x_{3}\}})^t,x_{3})$, where  $K=I(P_\omega^n\setminus{x_1})$. There are also  two subcases.

(a) If $\omega_{i}>\omega_{i+2} \ge 1$, then
by Theorem \ref{path} and  the inductive hypothesis, we can conclude that
\begin{eqnarray*}
\reg(S/(I^t,x_{2}))&\le& A+\lfloor \frac{i-3}{3} \rfloor+\lfloor \frac{n-(i+1)}{3} \rfloor,\\
\reg(S/(I^t : x_{2}): x_1)&\le&(A-2\omega)+\lfloor \frac{i-1}{3} \rfloor+\lfloor \frac{n-(i+1)}{3} \rfloor,\\
\reg(S/(K^t : x_{2}) : x_{3})&\le&(A-2\omega)+\lfloor \frac{i-2}{3} \rfloor+\lfloor \frac{n-(i+1)}{3} \rfloor,\\
\reg(S/((K^t : x_{2}),x_{3}))&\le&\begin{cases}
            A+\lfloor \frac{n-7}{3} \rfloor,             & \text{if $i=3$}        \\
            A+\lfloor \frac{i-4}{3} \rfloor+\lfloor \frac{n-(i+1)}{3} \rfloor, & \text{if $i<3$}\hspace{1.0cm}(17)\\
        \end{cases}
\end{eqnarray*}

where $A=(2\omega-1)+2\omega(t-1)$.

(b) If $\omega_i=\omega_{i+2}\ge 2$, then by Theorem \ref{path} and the inductive hypothesis, we have
\begin{footnotesize}
\begin{eqnarray*}
\reg\left(\frac{S}{(I^t,x_{2})}\right) &\le& A+\max\{\lfloor \frac{i-3}{3} \rfloor+\lfloor \frac{n-(i+1)}{3} \rfloor, \lfloor \frac{i-4}{3} \rfloor+\lfloor \frac{n-i}{3} \rfloor\}\\
 &\le& \reg(S/I)+2(t-1)\omega  \\
\reg\left(\frac{S}{(I^t : x_{2}): x_1}\right) &\le& (A-2\omega)+\max\{\lfloor \frac{i-1}{3} \rfloor+\lfloor \frac{n-(i+1)}{3} \rfloor, \lfloor \frac{i-2}{3} \rfloor+\lfloor \frac{n-i}{3} \rfloor\}\\
 &\le& \reg(S/I)+2(t-1)\omega-2\omega \\
\reg\left(\frac{S}{(K^t : x_{2}) : x_{3}}\right) &\le& (A-2\omega)+\max\{\lfloor \frac{i-2}{3} \rfloor+\lfloor \frac{n-(i+1)}{3} \rfloor,\lfloor \frac{i-3}{3} \rfloor+\lfloor \frac{n-i}{3} \rfloor\} \\
 &\le& \reg(S/I)+2(t-1)\omega-2\omega \\
\reg\left(\frac{S}{(K^t : x_{2}),x_{3}}\right) &\le& \begin{cases}
            A+\lfloor \frac{n-6}{3} \rfloor,             & \text{if $i=3$}        \\
            A+\max\{\lfloor \frac{i-4}{3} \rfloor+\lfloor \frac{n-(i+1)}{3} \rfloor,\lfloor \frac{i-5}{3} \rfloor+\lfloor \frac{n-i}{3} \rfloor\}, & \text{if $i<3$}\\
        \end{cases}  \\
&\le&  \reg(S/I)+2(t-1)\omega-1.    \hspace{5.5cm}(18)
\end{eqnarray*}
\end{footnotesize}

In the two cases mentioned above, by applying formulas (17) and  (18) as well as  Lemma \ref{exact}  to the  following exact sequences
\begin{gather*}
\hspace{1cm}\begin{matrix}
 0 & \rightarrow & \frac{S}{I^t : x_{2}}(-1)  & \stackrel{ \cdot x_{2}} \longrightarrow  & \frac{S}{I^t} & \rightarrow & \frac{S}{(I^t, x_{2})} & \rightarrow & 0, \\
  0 & \rightarrow & \frac{S}{(I^t : x_{2}) : x_1}(-1) & \stackrel{ \cdot  x_1} \longrightarrow & \frac{S}{I^t : x_{2}} &\rightarrow & \frac{S}{((I^t : x_{2}),x_1)} & \rightarrow & 0,\\
 0 & \rightarrow & \frac{S}{(K^t : x_{2}) : x_{3}}(-1) & \stackrel{ \cdot  x_{3}} \longrightarrow & \frac{S}{K^t : x_{2}} &\rightarrow & \frac{S}{((K^t : x_{2}),x_{3})} & \rightarrow & 0,
 \end{matrix}
\end{gather*}
It is always possible to ensure that $\reg (S/I^t) \le \reg(S/I)+2(t-1)\omega$.$\hfill$
\end{proof}

\begin{Theorem}
Let  $P_{\omega}^n$ be a path as in Remark \ref{n-path}. Then,    for any $t\ge 1$, we have
\[
\reg(S/I(P_{\omega}^n)^t)= \reg(S/I(P_{\omega}^n))+2(t-1)\omega
\]
where $\omega=\max\{\omega_i\mid i\in [n-1]\}$.
\end{Theorem}

\begin{proof} Let $I=I(P_{\omega}^n)$ and $\omega=\omega_i$, then $\omega_i\ge 2$. By applying symmetry, we can assume that
(a) $\omega_i > \omega_{i+2} \ge 1$, or (b) $\omega_i=\omega_{i+2} \ge 2$, and $\omega_j=1$ with $j \neq i,i+2$. We distinguish
between the following two cases:

(a)  Suppose  $\omega_{i}>\omega_{i+2} \ge 1$ and $\omega_j=1$ for all $j \neq i,i+2$. In this case,  let $(I^t)^{\calP}$ be the polarization of $I^t$, then  by Lemma \ref{spliting},  $(I^t)^{\calP}=J^{\calP}+K^{\calP}$, which is Betti splitting, and $J^{\calP} \cap K^{\calP}=J^{\calP}L^{\calP}$, where  $\mathcal{G}(J)=\{(x_ix_{i+1})^{t\omega_i}\}$, $\mathcal{G}(K)=\mathcal{G}(I^t) \setminus \mathcal{G}(J)$ and $L=(x_1x_2,\ldots,x_{i-3}x_{i-2})+(x_{i+3}x_{i+4},\ldots,x_{n-1}x_n)+(x_{i-1},x_{i+2})$ with $x_i=0$ if $i\le 0$. By Lemma \ref{sum2}, Lemma \ref{polar}  and Theorem
\ref{path},   it follows  that $\reg(J^{\calP})=\reg(J)=2t\omega$, 
\begin{align*}
\reg(J^{\calP} \cap K^{\calP})&=\reg(J^{\calP}L^{\calP})=\reg(J)+\reg(L)\\
&=2t\omega+(\lfloor \frac{i-1}{3} \rfloor+\lfloor \frac{n-(i+1)}{3} \rfloor+1)\\
&=\reg(I)+2(t-1)\omega+1.
\end{align*}
Let $H$ and $H'$ be hypergraphs associated with $\mathcal{G}((I^t)^{\calP})$ and $\mathcal{G}(K^{\calP})$, respectively. Then
$H'$ is an induced subhypergraph of $H$. Therefore, 
\[
\reg(K^{\calP}) \le \reg((I^t)^{\calP}) \le\reg(I)+2(t-1)\omega
\]
by Theorem \ref{path power4}.
Based on Corollary \ref{cor1} and Lemmas
\ref{quotient}  and \ref{polar}, we can conclude that 
\begin{align*}
\reg(S/I^t)&=\reg((I^t)^{\calP})-1\\
&=\max\{\reg(J^{\calP}),\reg(K^{\calP}),\reg(J^{\calP}\cap K^{\calP})-1\}-1\\
&=2(t-1)\omega+\reg(S/I).
\end{align*}
(b) If $\omega_i=\omega_{i+2} \ge 2$, then $\reg(S/I)=\max \{2\omega+\lfloor \frac{i-1}{3} \rfloor+\lfloor \frac{n-(i+1)}{3} \rfloor, 2\omega+\lfloor \frac{i-2}{3} \rfloor+\lfloor \frac{n-i}{3} \rfloor\}-1$ by Theorem \ref{path}. We have two subcases:

(i) If $\reg(S/I)=(2\omega-1)+\lfloor \frac{i-1}{3} \rfloor+\lfloor \frac{n-(i+1)}{3} \rfloor$, then the statements can be proved by arguments similar to part (a),
we omit the details.

(ii) If  $\reg(S/I)=(2\omega-1)+\lfloor \frac{i-2}{3} \rfloor+\lfloor \frac{n-i}{3} \rfloor$, then let $(I^t)^{\calP}$ be the polarization of $I^t$. In this case,  we can deduce that  
$(I^t)^{\calP}=J^{\calP}+K^{\calP}$, which  is Betti splitting, where $\mathcal{G}(J)=\{(x_{i+2}x_{i+3})^{t\omega_{i+2}}\}$ and $\mathcal{G}(K)=\mathcal{G}(I^t) \setminus \mathcal{G}(J)$. 
By arguments similar to part (a), we can conclude that
\begin{align*}
\reg(J^{\calP})&=\reg(J)=2t\omega,\\
\reg(K^{\calP}) &\le \reg((I^t)^{\calP}) \le\reg(I)+2(t-1)\omega,\\
\reg(J^{\calP} \cap K^{\calP})&=\reg(J^{\calP}L^{\calP})=\reg(J)+\reg(L)\\
&=2t\omega+(\lfloor \frac{i-2}{3} \rfloor+\lfloor \frac{n-i}{3} \rfloor+1)\\
&=\reg(I)+2(t-1)\omega+1.
\end{align*}
It follows from Corollary \ref{cor1} and Lemmas
\ref{quotient}  and \ref{polar} that
\begin{align*}
\reg(S/I^t)&=reg((I^t)^{\calP})-1=\max\{\reg(J^{\calP}),\reg(K^{\calP}),\reg(J^{\calP}\cap K^{\calP})-1\}-1\\
&=2(t-1)\omega+\reg(S/I).
  \end{align*}$\hfill$
\end{proof}

We conclude this paper by presenting the lower bound on the depth of powers of the edge ideal of a path $P_{\omega}^n$ with $n\ge 5$, as described in Remark \ref{n-path}.
\begin{Theorem} \label{tdepth1}
Let  $P_{\omega}^n$ be a path as in Remark \ref{n-path}, where $n\ge 5$.  If $\omega_1>\omega_3$  and $\omega_i=1$ for any $i\neq 1,3$. Then for any $t\ge 1$,  the following holds:
\begin{itemize}
 \item[(1)] if $\omega_3=1$, then $\depth(S/I(P_{\omega}^n)^t) \ge \max\big\{\lceil \frac{n-t+1}{3} \rceil,1\big\}$.
\item[(2)] if $\omega_3>1$, then $\depth(S/I(P_{\omega}^n)^t) \ge \max\big\{\lceil \frac{n-t+1}{3} \rceil,2 \big\}$.
\end{itemize}
\end{Theorem}

\begin{proof}
Let $I=I(P_{\omega}^n)$ and  $J=I(P_\omega^n\setminus{x_n})$. Then $J$ is an ideal in $S_1$, where  
$S_1=\KK[x_1,\ldots,x_{n-1}]$. Furthermore, by Lemma \ref{pathcolon},  we  can see that   $(J^t: x_{n-1}):x_{n-2}=J^{t-1}$, $((J^t: x_{n-1}),x_{n-2})=(I(P_\omega^n\setminus {\{x_n,x_{n-2}\}})^t,x_{n-2})$, $(I^t: x_{n-1}): x_n=I^{t-1}$, $((I^t: x_{n-1}),x_n)=((I(P_\omega^n\setminus{x_n})^t: x_{n-1}),x_n)$ and $(I^t,x_{n-1})=(I(P_\omega^n\setminus{x_{n-1}})^t,x_{n-1})$.

We will prove the statements by induction on $n$ and $t$. There are two cases to consider.

(1) If $\omega_3=1$, then the base case where  $t=1$ follows from  Theorem \ref{path}. Now, we assume that $t\ge 2$.  By using  Lemma \ref{sum2}, Theorem  \ref{small2}, Proposition \ref{path power5}, and the inductive hypothesis,
we can deduce that
\begin{align*}
\depth(S_1/(J^t: x_{n-1}):x_{n-2})& \ge \max\Big\{\lceil \frac{n-t+1}{3} \rceil,1\Big\}, \\
\depth(S_1/((J^t: x_{n-1}),x_{n-2})) &\ge \max\Big\{\lceil \frac{n-t+1}{3} \rceil,1\Big\},\\
\depth(S/((I^t: x_{n-1}): x_n)) &\ge \max\Big\{\lceil \frac{n-t+2}{3} \rceil,1\Big\},\\
\depth(S/(I^t,x_{n-1})) &\ge \max\Big\{\lceil \frac{n-t+2}{3} \rceil,1\Big\}.
\end{align*}
 Applying Lemma \ref{exact} to the following exact sequences
 \begin{gather*}
\hspace{1cm}\begin{matrix}
 0 & \rightarrow & \frac{S}{I^t : x_{n-1}}(-1)  & \stackrel{ \cdot x_{n-1}} \longrightarrow  & \frac{S}{I^t} & \rightarrow & \frac{S}{(I^t, x_{n-1})} & \rightarrow & 0, \\
  0 & \rightarrow & \frac{S}{(I^t : x_{n-1}) : x_n}(-1) & \stackrel{ \cdot  x_n} \longrightarrow & \frac{S}{I^t : x_{n-1}} &\rightarrow & \frac{S}{((I^t : x_{n-1}),x_n)} & \rightarrow & 0, \\
 0 & \rightarrow & \frac{S_1}{(J^t : x_{n-1}) : x_{n-2}}(-1) & \stackrel{ \cdot  x_{n-2}} \longrightarrow & \frac{S_1}{J^t : x_{n-1}} &\rightarrow & \frac{S_1}{((J^t : x_{n-1}),x_{n-2})} & \rightarrow & 0, & \hspace{1cm}(19)
 \end{matrix}
\end{gather*}
it follows that   $\depth(S/I^t) \ge \max\{\lceil \frac{n-t+1}{3} \rceil,1\}$.

(2) Suppose $\omega_3>1$. There are the following two subcases to consider.

(i) If  $n=5$, or $6$, then $((I^t : x_n),x_{n-1})=(I(P_\omega^n\backslash{x_{n-1}})^t,x_{n-1})$ and $(I^t,x_n)=(I(P_\omega^n\setminus{x_n})^t,x_n)$. By applying  Theorem \ref{path}, Lemma \ref{sum2}, and the inductive hypothesis, we can conclude that
\begin{align*}
\depth(S/((I^t : x_n) : x_{n-1})) &\ge 2, \ \  \depth(S/I)=2,\\
\depth(S/((I^t : x_n),x_{n-1}))&\ge 2, \ \  \depth(S/(I^t,x_n)) \ge 2.
\end{align*}
Thus we obtain  $\depth(S/I^t) \ge 2$ by applying Lemma \ref{exact} to the following  exact sequences
\begin{gather*}
\hspace{1cm}\begin{matrix}
 0 & \rightarrow & \frac{S}{I^t : x_{n}}(-1)  & \stackrel{ \cdot x_{n}} \longrightarrow  & \frac{S}{I^t} & \rightarrow & \frac{S}{(I^t, x_{n})} & \rightarrow & 0, \\
  0 & \rightarrow & \frac{S}{(I^t : x_{n}) : x_{n-1}}(-1) & \stackrel{ \cdot  x_{n-1}} \longrightarrow & \frac{S}{I^t : x_{n}} &\rightarrow & \frac{S}{((I^t : x_{n}),x_{n-1})} & \rightarrow & 0. & \hspace{1.5cm}(20)
 \end{matrix}
\end{gather*}

(ii) Suppose that  $n\ge 7$, we can use  Lemma \ref{sum2},  Theorem \ref{path}, and the inductive hypothesis to conclude that
\begin{align*}
\depth(S_1/(J^t: x_{n-1}):x_{n-2})& \ge \max\Big\{\lceil \frac{n-t+1}{3} \rceil,2\Big\},\\
\depth(S_1/((J^t: x_{n-1}),x_{n-2})) &\ge \max\Big\{\lceil \frac{n-t+1}{3} \rceil,2\Big\},\\
\depth(S/((I^t: x_{n-1}): x_n)) &\ge \max\Big\{\lceil \frac{n-t+2}{3} \rceil,2\Big\},\\
 \depth(S/(I^t,x_{n-1})) &\ge \max\Big\{\lceil \frac{n-t+2}{3} \rceil,2\Big\}.
\end{align*}
The desired formulas can be  obtained by applying Lemma  \ref{exact} to the exact sequences (19) mentioned above.
\end{proof}

\begin{Theorem}\label{tdepth2}
Let  $P_{\omega}^n$ be a path  as in Remark \ref{n-path}, where $n\ge 5$. If $\omega_i>\omega_{i+2}$ for some $i\ge 1$, and $\omega_j=1$ for all $j\neq i,i+2$. Then, for any $t\ge2$, we have 
\begin{footnotesize}
\item[(1)] if $\omega_{i+2}=1$,  then $\depth\big(\frac{S}{I(P_{\omega}^n)^t}\big) \ge\begin{cases}
            \lceil \frac{n-1}{3} \rceil,  &\!\! \text{if $t=2$, $i\equiv 1\!\!\pmod 3$ and $n\equiv 2\!\!\pmod 3$,}        \\
            \max\Big\{\lceil \frac{n-t}{3} \rceil,1\Big\}, &\!\! \text{otherwise.}
        \end{cases}$\\
\item[(2)] if $\omega_{i+2}>1$, then $\depth\big(\frac{S}{I(P_{\omega}^n)^t}\big) \ge \max\Big\{\lceil \frac{n-t}{3} \rceil,2\Big\}$.
\end{footnotesize}
\end{Theorem}

\begin{proof} Using the notations from  the proof of Theorem \ref{tdepth1}, we can observe from Lemma \ref{pathcolon} that   $(J^t: x_{n-1}):x_{n-2}=J^{t-1}$, $(I^t: x_{n-1}): x_n=I^{t-1}$, $((J^t: x_{n-1}),x_{n-2})=(I(P_\omega^n\setminus {\{x_n,x_{n-2}\}})^t,x_{n-2})$, $((I^t: x_{n-1}),x_n)=((I(P_\omega^n\setminus{x_n})^t: x_{n-1}),x_n)$ and $(I^t,x_{n-1})=(I(P_\omega^n\setminus{x_{n-1}})^t,x_{n-1})$.

We will prove the statements by induction on $n$ and $t$.  The base case, where  $t=1$, follows from  Theorem \ref{path}. Now, we assume that $t\ge 2$.  There are two cases:

(1) If $\omega_{i+2}=1$, then we can assume  that $2 \le i \le \lfloor \frac{n}{2} \rfloor$ due to symmetry.  Since $\depth(S/I)=\min \{\lceil\frac{i}{3}\rceil+\lceil\frac{n-i}{3}\rceil, \lceil \frac{i-2}{3} \rceil+\lceil \frac{n-i-2}{3} \rceil+1\}\ge \lceil\frac{n-1}{3}\rceil$ by Theorem \ref{path}. Applying Lemma \ref{sum2}, Theorem \ref{tdepth1}, and the inductive hypothesis, we can conclude that

(i) If  $t=2$, $i\equiv 1\!\!\pmod 3$ and $n\equiv 2\!\!\pmod 3$, then 
\begin{align*}
\depth(S_1/(J^t : x_{n-1}x_{n-2})) &\ge \lceil \frac{n-1}{3} \rceil, \ \ \depth(S/(I^t : x_{n-1}x_{n})) \ge \lceil \frac{n-1}{3} \rceil,\\
\depth(S_1/((J^t : x_{n-1}),x_{n-2}))&\ge \lceil \frac{n-1}{3} \rceil,\ \  \depth(S/(I^t, x_{n-1}))\ge \lceil \frac{n-1}{3} \rceil.
\end{align*}

(ii) Otherwise,  we have
\begin{align*}
\depth(S_1/(J^t : x_{n-1}x_{n-2})) &\ge \max\big\{\lceil \frac{n-t}{3} \rceil,1\big\}, \\
\depth(S_1/((J^t : x_{n-1}),x_{n-2}))&\ge \max\big\{\lceil \frac{n-t}{3} \rceil,1\big\},\\
\depth(S/(I^t : x_{n-1}x_{n})) &\ge \max\big\{\lceil \frac{n-t+1}{3} \rceil,1\big\},\\
\depth(S/(I^t, x_{n-1}))&\ge \max\big\{\lceil \frac{n-t+1}{3} \rceil,1\big\}.
\end{align*}
By Lemma \ref{exact} and  the exact sequences
(19) mentioned above, we obtain  
\[
\depth\big(\frac{S}{I(P_{\omega}^n)^t}\big) \ge\begin{cases}
            \lceil \frac{n-1}{3} \rceil,  &\text{if $t=2$, $i\equiv 1\!\!\!\!\pmod 3$ and $n\equiv 2\!\!\!\!\pmod 3$,}        \\
            \max\Big\{\lceil \frac{n-t}{3} \rceil,1\Big\}, &\!\! \text{otherwise.}
        \end{cases}
\]

(2) If $\omega_{i+2}>1$,  then we can assume  that  $2 \le i \le \lfloor \frac{n}{2} \rfloor-1$ due to symmetry. There are the following two cases:

(i) If  $n=6$ or $7$, then  $((I^t : x_n),x_{n-1})=(I(P_\omega^n\backslash{x_{n-1}})^t,x_{n-1})$ and $(I^t,x_n)=(I(P_\omega^n\backslash{x_n})^t,x_n)$. Applying  Lemma \ref{sum2},   Theorem \ref{path}, Theorem \ref{tdepth1}, and the inductive hypothesis, it follows that
\begin{align*}
\depth(S/((I^t : x_n) : x_{n-1}))& \ge 2, \ \  \depth(S/I)=2,\\
\depth(S/((I^t : x_n),x_{n-1}))&\ge 2, \ \ \depth(S/(I^t,x_n)) \ge 2.
\end{align*}
 By Lemma \ref{exact} and the exact sequences (20) mentioned above,  we obtain  $\depth(S/I^t) \ge 2$,  which means that $\depth\big(\frac{S}{I(P_{\omega}^n)^t}\big) \ge \max\Big\{\lceil \frac{n-t}{3} \rceil,2\Big\}$.
 
 (ii) If  $n\ge 8$, then $\depth(S/I)=\min \{\lceil\frac{i}{3}\rceil+\lceil\frac{n-i-1}{3}\rceil, \lceil \frac{i-2}{3} \rceil+\lceil \frac{n-i-2}{3} \rceil+1\}\ge \lceil\frac{n-1}{3}\rceil$  by Theorem \ref{path}. 
Using Lemma \ref{sum2}, Theorem \ref{tdepth1}, and  the inductive hypothesis, we obtain 
\begin{align*}
\depth(S_1/(J^t : x_{n-1}x_{n-2})) &\ge \max\big\{\lceil \frac{n-t}{3} \rceil,2\big\},\\
\depth(S_1/((J^t : x_{n-1}),x_{n-2}))&\ge \max\big\{\lceil \frac{n-t}{3} \rceil,2\big\},\\
\depth(S/(I^t : x_{n-1}x_{n})) &\ge \max\big\{\lceil \frac{n-t+1}{3} \rceil,2\big\},\\
\depth(S/(I^t, x_{n-1}))&\ge \max\big\{\lceil \frac{n-t+1}{3} \rceil,2\big\}.
\end{align*}
Applying Lemma \ref{exact} to the exact sequences (19) mentioned above,  we obtain $\depth(S/I^t) \ge \max\big\{\lceil \frac{n-t}{3} \rceil,2\big\}$.
\end{proof}

The following four examples provide instances in which the lower bounds of powers of the edge ideal of a non-trivial edge-weighted integrally closed path in Theorems \ref{tdepth1} and \ref{tdepth2} are attained.
\begin{Example}
 $I_1=(x_1^2x_2^2,x_2x_3,x_3x_4,x_4x_5)$, $I_2=(x_1^2x_2^2,x_2x_3,x_3x_4,x_4x_5,x_5x_6)$ and $I_3=(x_1^2x_2^2,x_2x_3,x_3x_4,x_4x_5,x_5x_6,x_6x_7)$ are edge ideals of integrally closed paths
 $P_\omega^5$, $P_\omega^6$ and $P_\omega^7$, respectively.  Using CoCoA, we obtain that $\depth(R/I_1^2)=\depth(R/I_2^2)=\depth(R/I_3^2)=2$, which is  the lower bound given by Theorem \ref{tdepth1}.
\end{Example}

\begin{Example} $I_1=(x_1^2x_2^2,x_2x_3,x_3^3x_4^3,x_4x_5)$, $I_2=(x_1^4x_2^4,x_2x_3,x_3^2x_4^2,x_4x_5,x_5x_6)$ and $I_3=(x_1^2x_2^2,x_2x_3,x_3^3x_4^3,x_4x_5,x_5x_6,x_6x_7)$ are edge ideals of integrally closed paths
 $P_\omega^5$, $P_\omega^6$ and $P_\omega^7$, respectively.  Using CoCoA, we obtain that $\depth(R/I_1^2)=\depth(R/I_2^2)=\depth(R/I_3^2)=2$, which is  the lower bound given by Theorem \ref{tdepth1}.
\end{Example}

\begin{Example}
 $I_1=(x_1x_2,x_2^2x_3^2,x_3x_4,x_4x_5)$, $I_2=(x_1x_2,x_2^2x_3^2,x_3x_4,x_4x_5,x_5x_6)$, $I_3=(x_1x_2,x_2^2x_3^2,x_3x_4,x_4x_5,x_5x_6,x_6x_7)$  and $I_4=(x_1x_2,x_2x_3,x_3x_4,x_4x_5^2,x_5x_6,x_6x_7,x_7x_8)$ are edge ideals of integrally closed paths $P_\omega^5$, $P_\omega^6$, $P_\omega^7$ and $P_\omega^8$, respectively. Using  CoCoA, we obtain that $\depth(R/I_1^2)=1$, $\depth(R/I_2^2)=\depth(R/I_3^2)=2$ and $\depth(R/I_4^2) = 3$, which are  the lower bounds given by Theorem \ref{tdepth2}.
\end{Example}

\begin{Example}
 $I_1=(x_1x_2,x_2^2x_3^2,x_3x_4,x_4^3x_5^3,x_5x_6)$, $I_2=(x_1x_2,x_2^2x_3^2,x_3x_4,x_4^3x_5^3,x_5x_6,\\
 x_6x_7)$ and $I_3=(x_1x_2,x_2^2x_3^2,x_3x_4,x_4^3x_5^3,x_5x_6,x_6x_7,x_7x_8)$ are edge ideals of integrally closed paths
 $P_\omega^6$, $P_\omega^7$ and $P_\omega^8$, respectively. By using CoCoA, we obtain that $\depth(R/I_1^2)\\
 =\depth(R/I_2^2)=\depth(R/I_3^2)=2$, which is  the lower bound given by Theorem \ref{tdepth2}.
\end{Example}

\medskip
\hspace{-6mm} {\bf Acknowledgments}

 \vspace{3mm}
\hspace{-6mm}  This research is supported by the Natural Science Foundation of Jiangsu Province (No. BK20221353). We gratefully acknowledge the use of the computer algebra system CoCoA (\cite{Co}) for our experiments.






\begin{thebibliography}{99}
\bibitem{B}  A. Banerjee, The regularity of powers of edge ideals, {\it J. Algebraic Combin.}, 41 (2014), 303-321.




\bibitem{BHT} S. Beyarslan, H. T. H\`a and T. N. Trung, Regularity of powers of forests and cycles, {\it J. Algebraic Combin.}, 42 (2015), 1077-1095.

\bibitem{Br}  M. Brodmann,  The asymptotic nature of the analytic spread, {\it  Math. Proc. Cambridge Philos Soc.}, 86 (1979), 35--39.

\bibitem{BH}  W. Bruns and J. Herzog, {\it Cohen–Macaulay Rings}, Revised  edition. (Cambridge University Press, 1998).



\bibitem{Co} CoCoATeam, CoCoA: a system for doing Computations in Commutative Algebra, Avaible at
http://cocoa.dima.unige.it.

\bibitem{Con} A. Conca, Regularity jumps for powers of ideals,   In Commutative algebra, volume 244 of  {\it Lect. Notes Pure Appl. Math.},  (2006), 21--32.


\bibitem{CHT} S. D. Cutkosky, J. Herzog and N. V. Trung, Asymptotic behaviour of the Castelnuovo-Mumford regularity, {\it Compos. Math.} 118 (3) (1999), 243-261.

\bibitem{DMV}  L. T. K. Diem, N. C. Minh, and T. Vu, The sequentially Cohen-Macaulay property of edge ideals of edge-weighted graphs, arXiv:2308.05020.


\bibitem{DZCL}  S. Y. Duan, G. J. Zhu, Y. J. Cui and J. X. Li, Integral closure and normality of edge ideals of some edge-weighted graphs, arXiv:2308.06016.








\bibitem{EK} S. Eliahou and M. Kervaire, Minimal resolutions of some monomial ideals, {\it J. Algebra}, 129 (1990), 1-25.

\bibitem{SF} S. Faridi, Monomial ideals via square-free monomial ideals, {\it  Comm. Algebra}, 244 (2006), 85-114.

\bibitem{SM} S. Morey, Depths of powers of the edge ideal of a tree, {\it  Comm. Algebra}, 38(11) (2010), 4042-4055.

\bibitem{F} G. Fatabbi, On the resolution of ideals of fat points, {\it J. Algebra}, 242 (2001), 92-108.


\bibitem{FHT} C. A. Francisco, H. T. H\`a and A. Van Tuyl, Splittings of monomial ideals, {\it Proc. Amer. Math. Soc.}, 137 (10) (2009),
3271-3282.

\bibitem{MRW} W. Frank Moore, Mark Rogers, and Keri Sather-Wagstaff. Monomial ideals and their decompositions. Universi-text. Springer, Cham, 2018.

\bibitem{FM} L. Fouli and  S. Morey,  A lower bound for depths of powers of edge ideals, {\it J. Algebr Comb}, 42 (2015), 829-848.


\bibitem{HT1} H. T. H\`a  and T. N. Trung, Depth and regularity of powers of sums of ideals, {\it Math. Z.}, 282 (2016), 819-838.




\bibitem{HT} Jing He and A. Van Tuyl, Algebraic properties of the path ideal of a tree,  {\it Comm. Algebra}, 38 (5) (2010), 1725-1742.

\bibitem{HH} J. Herzog and T. Hibi,  {\it  Monomial ideals}, New York, NY, USA: Springer-Verlag, 2011.

\bibitem{HHi} J. Herzog and T. Hibi,  The depth of powers of an ideal, {\it J. Algebra}, 291 (2) (2005), 534--550.



\bibitem{Hi}  T. T. Hien, Cohen-Macaulay edge-weighted graphs of girth 5 or greater, arXiv:2309.05056.

\bibitem{HT2} L. T. Hoa and N. D. Tam, On some invariants of a mixed product of ideals, {\it Arch. Math}, 94 (4) (2010), 327-337.

\bibitem{SJ} S. Jacques, Betti numbers of graph ideals, arXiv:math/0410107.

\bibitem{K} V. Kodiyalam, Asymptotic behaviour of Castelnuovo-Mumford regularity, {\it Proc. Amer. Math. Soc.}, 128 (1999), 407--411.

\bibitem{MTV} N. C. Minh, T. N. Trung and T. Vu, Depth of powers of edge ideals of cycles and trees, arXiv:2308.00874.



\bibitem{M}  S. Morey, Depths of powers of the edge ideal of a tree, {\it Comm. Algebra}, 38, (2010),  4042-4055.



\bibitem{MV} S. Morey and R. H. Villarreal, {\it Edge ideals: algebraic and combinatorial properties}, in Progress
in Commutative Algebra, Combinatorics and Homology, Vol. 1 (C. Francisco, L. C. Klingler,
S. Sather-Wagstaff, and J. C. Vassilev, Eds.), De Gruyter, Berlin, 2012, 85-126.




\bibitem{PS} C. Paulsen and S. Sather-Wagstaff, Edge ideals of weighted graphs, {\it J. Algebra Appl.}, 12 (2013), 1250223-1-24.

\bibitem{FSTY} S. A. Seyed Fakhari, K. Shibata, N. Terai and  Siamak Yassemi, Cohen-Macaulay edge weighted edge ideals of very well-covered graphs,
 {\it Comm. Algebra} 49(10) (2021),  4249-4257.

\bibitem{W} S.  Wei, Cohen-Macaulay weighted chordal graphs, arXiv:2309.02414.


\bibitem{Z}  G. J. Zhu, Projective dimension and regularity of the path ideal of the line graph, {\it J. Algebra Appl.}, 17(4), (2018), 1850068-1-15.














 \end{thebibliography}
\end{document}